\immediate\write18{biber \jobname}
\documentclass{amsart}
\usepackage{microtype}
\usepackage{amssymb}
\usepackage{mathtools}
\usepackage{tikz-cd}
\usepackage{mathbbol} % changes \mathbb{} and adds more support
\usepackage[
	backend=biber,
	style=alphabetic,
	backref=true,
	url=false,
	doi=false,
	isbn=false,
	eprint=false]{biblatex}

\renewbibmacro{in:}{} % don't display "in:" before the journal name
\AtEveryBibitem{\clearfield{pages}} % don't show page numbers

\DeclareFieldFormat{title}{\myhref{\mkbibemph{#1}}}
\DeclareFieldFormat
[article,inbook,incollection,inproceedings,patent,thesis,unpublished]
{title}{\myhref{\mkbibquote{#1\isdot}}}

\newcommand{\doiorurl}{%
	\iffieldundef{url}
	{\iffieldundef{eprint}
		{}
		{http://arxiv.org/abs/\strfield{eprint}}}
	{\strfield{url}}%
}

\newcommand{\myhref}[1]{%
	\ifboolexpr{%
		test {\ifhyperref}
		and
		not test {\iftoggle{bbx:eprint}}
		and
		not test {\iftoggle{bbx:url}}
	}
	{\href{\doiorurl}{#1}}
	{#1}%
}

\usepackage[
	bookmarks=true,
	linktocpage=true,
	bookmarksnumbered=true,
	breaklinks=true,
	pdfstartview=FitH,
	hyperfigures=false,
	plainpages=false,
	naturalnames=true,
	colorlinks=true,
	pagebackref=false,
	pdfpagelabels]{hyperref}

\hypersetup{
	colorlinks,
	citecolor=blue,
	filecolor=blue,
	linkcolor=blue,
	urlcolor=blue
}

\usepackage[capitalize, noabbrev]{cleveref}
\crefname{subsection}{\S\!}{subsections}

\calclayout

\makeatletter
\@namedef{subjclassname@2020}{%
	\textup{2020} Mathematics Subject Classification}
\makeatother

\theoremstyle{plain}
\newtheorem{theorem}{Theorem}[section]
\newtheorem{proposition}[theorem]{Proposition}
\newtheorem{lemma}[theorem]{Lemma}
\newtheorem{corollary}[theorem]{Corollary}
\newtheorem{question}[theorem]{Question}
\newtheorem*{question*}{Question}

\theoremstyle{definition}
\newtheorem{definition}[theorem]{Definition}
\newtheorem*{definition*}{Definition}
\newtheorem{construction}[theorem]{Construction}

\theoremstyle{remark}
\newtheorem{remark}[theorem]{Remark}
\newtheorem*{remark*}{Remark}
\newtheorem{example}[theorem]{Example}
\newtheorem*{example*}{Example}
\newtheorem{convention}[theorem]{Convention}
\newtheorem*{convention*}{Convention}
\newtheorem{notation}[theorem]{Notation}
\newtheorem*{notation*}{Notation}

\DeclareMathOperator{\bd}{\partial}
\newcommand{\ot}{\otimes}

\newcommand{\Z}{\mathbb{Z}}

\newcommand{\Sym}{\mathbb{S}}
\newcommand{\Cyc}{\mathbb{C}}

\newcommand{\Set}{\mathsf{Set}}
\newcommand{\Setp}{\mathsf{Set}_*}

\newcommand{\Mod}[1]{\mathsf{Mod}(#1)}

\newcommand{\Ch}[1]{\mathsf{Ch}(#1)}
\newcommand{\simplex}{\triangle}

\newcommand{\dd}{\vee}
\newcommand{\chains}[1][*]{\operatorname{N}_{#1}}
\newcommand{\cochains}[1][*]{\operatorname{N}^\dd_{#1}}
\newcommand{\uchains}[1][*]{\operatorname{C}_{#1}}

\newcommand{\rchains}[1][*]{\operatorname{N}^{r}_{#1}}

\newcommand{\cadenas}[1][*]{N_{#1}}
\newcommand{\ucadenas}[1][*]{C_{#1}}
\newcommand{\cocadenas}[1][*]{N^\dd_{#1}}
\newcommand{\ucocadenas}[1][*]{C^\dd_{#1}}
\newcommand{\rcadenas}[1][*]{N^{r}_{#1}}
\newcommand{\rucadenas}[1][*]{C^{r}_{#1}}

\newcommand{\rA}{\operatorname{L}}

\newcommand{\abel}{\operatorname{{\it L}}}

\newcommand{\Omd}[1][*]{\operatorname{\Omega}^{\dd}_{#1}}
\newcommand{\Om}[1][*]{\operatorname{\Omega}_{#1}}
\newcommand{\Omhatd}[1][*]{\hat{\operatorname{\Omega}}^{\dd}_{#1}}
\newcommand{\Omhat}[1][*]{\hat{\operatorname{\Omega}}_{#1}}

\newcommand{\rW}[1][*]{\operatorname{W}_{#1}}
\newcommand{\rWd}[1][*]{\operatorname{W}^\dd_{#1}}
\newcommand{\rWhat}[1][*]{\hat{\operatorname{W}}_{#1}}
\newcommand{\rWhatd}[1][*]{\hat{\operatorname{W}}^{\dd}_{#1}}
\newcommand{\W}[1][*]{W_{#1}}
\newcommand{\Wd}[1][*]{W^\dd_{#1}}
\newcommand{\Wst}[1][*]{W^{\mathrm{st}}_{#1}}

\DeclarePairedDelimiter\bars{\lvert}{\rvert}

\DeclarePairedDelimiter\set{\{}{\}}

\newcommand{\id}{\mathsf{id}}

\newcommand{\op}{{\mathrm{op}}}
\DeclareMathOperator*{\colim}{colim}

\newcommand{\xra}[1]{\xrightarrow{#1}}

\newcommand{\bF}{\mathbb{F}}

\newcommand{\bZ}{\mathbb{Z}}

\newcommand{\cC}{\mathcal{C}}

\newcommand{\cE}{\mathcal{E}}

\newcommand{\cP}{\mathcal{P}}

\newcommand{\sign}[1]{\operatorname{sign}(#1)}
\newcommand{\perm}{\pi}

\DeclareMathOperator{\DDD}{D}
\newcommand{\Psiom}{\operatorname{\check{\Psi}}}
\newcommand{\alex}{\Lambda}
\newcommand{\twist}{\operatorname{twist}}
\DeclareMathOperator{\barra}{|}
\DeclareMathOperator{\join}{\ast}

\DeclareMathOperator{\power}{P}

\DeclareMathOperator{\hotimes}{\hat{\otimes}}
\newcommand{\nf}{nf}
\newcommand{\f}{f}

\newcommand{\green}[1]{{\color{green} #1}}

\renewcommand{\d}{\operatorname{d}} %face maps in a simplicial/cosimplicial object
\newcommand{\diff}{d} %the face map that is part of the differential of the chain complex of a simplicial object.
\newcommand{\s}{\operatorname{s}} %degeneracy maps in a simplicial/cosimplicial object

\newcommand{\signo}{\mathfrak{s}}
\newcommand{\signot}{\mathfrak{t}}

\newcommand{\EG}{\mathrm{E}G}
\newcommand{\EC}{\mathrm{E}\Cyc}
\newcommand{\ES}{\mathrm{E}\Sym}
\newcommand{\ESst}{\mathrm{E}\Sym^{\mathrm{st}}}
\newcommand{\sd}{\operatorname{sd}}
\newcommand{\Psd}{\operatorname{P}}
\newcommand{\PsdSimp}{\Psd(\partial\asimplex^n)}

\newcommand{\Id}{\id}

\newcommand{\prev}[1]{}

\newcommand{\Sq}{\operatorname{Sq}}

\newcommand{\lra}{\longrightarrow}
\newcommand{\lla}{\longleftarrow}

\DeclareMathOperator{\hofib}{hofib}

\newcommand{\asimplex}{\triangle_{+}}
\newcommand{\kansimplex}{\triangle_{st}}

\newcommand{\sus}[1]{\operatorname{s}^{#1}}
\newcommand{\susp}[1]{\Sigma^{#1}}
\DeclareMathOperator{\ob}{ob}

\usepackage{appendix}
\usetikzlibrary{calc}
\usepackage[all]{xy}
\usepackage{tikz}
\usepackage{marginnote}

\usepackage{todo}
\title{Connected power operations and simplicial Poincar\'e duality}
\author[Cantero-Mor\'an]{Federico~Cantero-Mor\'an}
\address{Departamento de Matem\'aticas, Universidad Aut\'onoma de Madrid \& ICMAT. Calle Francisco Tom\'as y Valiente, 7. 28049, Madrid, Spain.}
\email{\href{mailto:federico.cantero@uam.es}{federico.cantero@uam.es}}
\thanks{F.C. was supported by Grants SI3/PJI/2021-00505 from Comunidad de Madrid and PID2019-108936GB-C21 from the Spanish government.}

\author[Medina-Mardones]{Anibal~M.~Medina-Mardones}
\address{Department of Mathematics, Western University, Canada}
\email{\href{mailto:anibal.medina.mardones@uwo.ca}{anibal.medina.mardones@uwo.ca}}
\thanks{A.M. was partially supported for this project through grant ANR 20 CE40 0016 01 PROJET HighAGT}

\date{\today}

\begin{document}
	\begin{center}
		||PREPRINT||
		\vskip 20pt
	\end{center}
	% !TEX root = ../oddp.tex

% \begin{abstract}
%     We define a structure (``connected cyclic diagonal'') on a chain complex that produces stable power operations in its cohomology with the property that the negative power operations always vanish. This structure is designed to encode power operations for spectra, whose cohomology lacks a cup product. We then use this structure to define cochain level operations for augmented simplicial objects. The main tool to build them is a duality on the integral chains of the standard augmented simplex.
% \end{abstract}

\begin{abstract}
    We introduce a structure termed ``connected cyclic diagonal" on a chain complex, which induces stable power operations in its cohomology with the property that negative power operations consistently vanish. This chain level structure is useful to represent power operations for spectra, whose cohomology lacks a cup product. Using a Poincar\'e duality algebra structure on the integral chains of the standard augmented simplex, we provide an effective method for the construction of cyclic diagonals on a broad class of augmented simplicial objects.
\end{abstract}
	\maketitle
	\tableofcontents
 \newpage
	% !TEX root = ../oddp.tex

\section{Introduction}\label{s:introduction}

\subsection{Steenrod power operations}

In 1946, Steenrod was looking for generalizations of the Hopf invariant when he found for $i \geq 0$ certain natural linear maps
\[
\Delta_i \colon \ucadenas(X) \to \ucadenas(X) \ot \ucadenas(X)
\]
on the singular chains of a space $X$, which extended the Alexander--Whitney diagonal $\Delta_0$.
These so-called \textit{cup-$i$ coproducts} satisfy $\bd\Delta_i + \Delta_i\bd = (1+\rho)\Delta_{i-1}$ for $i > 0$, where $\rho$ is the transposition of tensor factors, and give rise in mod~2 cohomology to the celebrated \emph{Steenrod squares}
\begin{equation}\label{eq:intro Sq}
	\begin{tikzcd}[column sep=small, row sep=0]
		\Sq^{k} \colon &[-15pt] H^*(X;\mathbb{F}_2) \arrow[r] & H^{*+k}(X;\mathbb{F}_2) \\
		& {[\alpha]} \arrow[r, mapsto] & {\big[(\alpha \otimes \alpha)\Delta_{k-\bars{\alpha}}\big]}.
	\end{tikzcd}
\end{equation}
The maps $\Delta_i$ effectively introduced by Steenrod in \cite{steenrod1947products} have been given alternative formulas in \cite{real1996computability,medina2021fast_sq} and are axiomatically characterized \cite{medina2022axiomatic}.
They can be reinterpreted as a single $\Cyc_2$-equivariant chain map
\[
\mu \colon W_*(2) \ot \ucadenas(X) \to \ucadenas(X)^{\ot 2},
\]
where $W_*(2)$ is the minimal free resolution of the ground ring as a module over the group ring of $\Cyc_2$, the cyclic group of order two.

Using the acyclic carrier theorem, Steenrod showed in \cite{steenrod1952reduced} the existence for each $r$ of an $\Sym_r$-equivariant chain map
\[
\mu \colon V_\ast(r) \ot \ucadenas(X) \to \ucadenas(X)^{\ot r}
\]
where $V_\ast(r)$ is a free resolution of the base ring as module over the group ring of the symmetric group $\Sym_r$, and used it to define a cohomology operation for each class in the homology of $\Sym_r$.

In \cite{steenrod1953cyclic}, Steenrod concentrated in those operations coming from classes induced by an inclusion $\Cyc_r \leq \Sym_r$ of the cyclic group of order $r$, with $r$ assumed to be a prime.
In this case, the inclusion is surjective in homology with mod~$r$ coefficients and the resulting operations in the mod~$r$ cohomology of a space are known as \textit{Steenrod power operations}.

\subsection{Unstable diagonals}

In \cite{may1970general}, May gave a general approach to power operations including Steenrod's power operations \cite{steenrod1962cohomology} and Dyer--Lashof operations on the homology of infinite loop spaces \cite{dyer62lashof}.
The basic structure used in \cite{may1970general} for the construction of power operations on a chain complex $\ucadenas$ is a $\Cyc_r$-equivariant chain map
\[
\mu \colon W_*(r) \ot \ucadenas \to \ucadenas^{\ot r}
\]
where $W_*(r)$ is the minimal free resolution of the ground ring $R$ as a module over the group ring of the cyclic group $\Cyc_r$:
\[
\Cyc_r\langle e_0\rangle
\overset{T}{\lla}
\Cyc_r\langle e_1\rangle 
\overset{N}{\lla}
\Cyc_r\langle e_2\rangle \overset{T}{\lla} 
\dotsb 
\]
We will refer to $\mu$ as an \textit{unstable $r$-cyclic diagonal} on $\ucadenas$, or simply an \textit{unstable $r$-diagonal}, saying that it is \textit{May--Steenrod} if it factors as
\[
\mu \colon W_*(r) \ot \ucadenas \xra{f \ot \id} V_*(r) \ot \ucadenas \xra{\nu} \ucadenas^{\ot r}
\]
with $f$ and $\nu$ a $\Cyc_r$- and a $\Sym_r$-equivariant map respectively.
We denote the maps $\mu(e_i \ot -) \colon \ucadenas \to \ucadenas^{\ot r}$ by $\Delta_{r,i}$ and refer to them as \textit{cup-$(r,i)$ coproducts}.

Generalizing the definition of $\Sq^i$ in \cref{eq:intro Sq} to a general prime $r$, a chain complex $\ucadenas$ with an unstable $r$-diagonal has power operations on its cohomology defined by
\begin{equation}\label{eq:intro P}
	\begin{tikzcd}[column sep=tiny, row sep=0]
		\power^i \colon &[-15pt] H^*(\ucadenas;\bF_r) \arrow[r] & H^{*+i}(\ucadenas;\bF_r) \\
		& {[\alpha]} \arrow[r, mapsto] & \big{[}(\alpha \ot \overset{r}{\dots} \ot \alpha)\Delta_{(r-1)\bars{\alpha}-i}\big{]},
	\end{tikzcd}
\end{equation}
where we have suppressed certain invertible coefficient.
We remark that the grading of this operations is different from the one used by Steenrod and May.

If the prime $r$ is odd and $\mu$ is May--Steenrod then
\[
\power^i = 0 \text{ unless } i = 2k(r-1) \text{ or }i = 2k(r-1)+1.
\]
Additionally, these operations often satisfy
\[
\power^i(x) = 0 \text{ for negative } i.
\]
For instance, this is the case with spaces.
Generally, verifying this property in practice is not straightforward.

An effective construction of cup-$(r,i)$ coproducts for simplicial sets depending solely on the combinatorics of simplices was introduced in \cite{medina2021may_st} and implemented in \cite{medina2021comch}.

\subsection{Stable diagonals}

Similarly to how Steenrod power operations on the cohomology of spaces motivated the definition of an unstable $r$-diagonal on a chain complex, we can consider spectra to motivate the definition of stable $r$-diagonals. Let $E = \{E_m\}_{m \geq 0}$ be a suspension spectrum. Its chains are given by
\[
\ucadenas(E) = \colim \left(\ucadenas(E_0) \to \dots \to \Sigma^{-m}\ucadenas(E_m) \to \Sigma^{-m-1} \ucadenas(E_{m+1}) \to \dotsb\right).
\]
The homotopy fibre of the norm map
\[
\W(r) \overset{N}{\lra} \Wd(r)
\]
is isomorphic to the acyclic unbounded chain complex $\Wst(r)$ defined as
\[
\dotsb \lla \Cyc_r\langle e_{-2}\rangle
\overset{T}{\lla}
\Cyc_r\langle e_{-1}\rangle
\overset{N}{\lla}
\Cyc_r\langle e_0\rangle
\overset{T}{\lla}
\Cyc_r\langle e_{1}\rangle
\overset{N}{\lla}
\Cyc_r\langle e_{2}\rangle
\overset{T}{\lla}
\dotsb,
\]
Any unstable diagonal suitably compatible with suspensions gives rise to an $\Cyc_r$-equivariant chain map
\[
\Wst(r) \ot \ucadenas(E) \lra \ucadenas(E)^{\ot r},
\]
which factors through a stable version $V_*^\mathrm{st}(r)$ of $V_*(r)$. We will see in Section \ref{s:suspension} that the effective construction of \cite{medina2021may_st} is indeed compatible with suspension.

Abstracting this example we define a \emph{stable $r$-cyclic diagonal} on a chain complex $\ucadenas$, or simply a \textit{stable $r$-diagonal} on $\ucadenas$, as a $\Cyc_r$-equivariant chain map
\[
\mu \colon \Wst(r) \ot \ucadenas \lra \ucadenas^{\ot r}.
\]
We say $\mu$ is \textit{May--Steenrod} if it factors through $V_*^\mathrm{st}$ as in the previous section.

The same definition of power operations given in \cref{eq:intro P} applies in this context.

\subsection{Connected diagonals}

The negative power operations of a stable $r$-diagonal $\mu$ are supported on a certain graded chain subcomplex $A$ of the bicomplex $\Wst \otimes \ucadenas$. We introduce the notion of \emph{connected $r$-cyclic diagonal}, or \textit{connected $r$-diagonal} for short, as a stable $r$-diagonal that vanishes on that subcomplex. In particular, its power operations are such that $\power^i = 0$ for $i<0$.

Let $\rWd(r)$ be the augmented dual of the right suspension of $\W(r)$:
\[
\dotsb
\Cyc_r\langle e_{-3}^\dd\rangle
\overset{-N}{\lla}
\Cyc_r\langle e_{-2}^\dd\rangle
\overset{T}{\lla}
\Cyc_r\langle e_{-1}^{\dd}\rangle
\overset{-N}{\lla}
R
\]
We will define a certain chain complex $\rWd(r)\hotimes \rucadenas$ isomorphic as a $\Cyc_r$-module to $\rWd(r)\otimes \ucadenas$, and isomorphic as a bigraded $\Cyc_r$-module to the quotient of $\Wst(r)\otimes \ucadenas$ by the subcomplex $A$, so that a connected $r$-diagonal is the same as a chain map
\[
\mu\colon \rWd(r)\hotimes \rucadenas\lra \ucadenas^{\otimes r}.
\]

\subsection{Connected diagonals for simplicial objects}

Let $\cC$ be a monoidal category with a natural diagonal $\Delta \colon X \to X^{\ot r}$ for each object $X$, and a strong monoidal functor $\abel \colon \cC \to \Mod{R}$ to the category of $R$-modules.
By virtue of $L$, every simplicial object in $\cC$ has an associated simplicial $R$-module, from which one gets a chain complex.
The main examples of $\cC$ are the categories $\Set$ and $\Setp$ of sets and pointed sets.
Our main result is an effective construction of connected $r$-diagonals from the following combinatorial structure.

\begin{definition}
	An \emph{$r$-cyclic straightening with duality} is a $\Cyc_r$-equivariant choice of element $x_\tau\in \tau$ for every proper non-empty subset $\tau \subset \{0,1,\dots,r-1\}$ such that the predecessor of $x_\tau$ is not in $\tau$.
\end{definition}

We mention that this structure can exist only if $r$ is prime.

\begin{theorem}\label{thm:main}
	If $r$ is an odd prime, each $r$-cyclic straightening with duality yields a natural connected $r$-diagonal on the chain complex of an augmented simplicial object in $\cC$.
\end{theorem}

If $r=2$ the theorem still holds, but only with coefficients in $\bF_2$.
In that case there is a unique $r$-cyclic straightening with duality, and therefore our theorem yields a unique connected diagonal.
When applied to a non-augmented simplicial set, it coincides with the unstable diagonal axiomatized in \cite{medina2022axiomatic}.
If $r = 3$ there is again a unique $r$-cyclic straightening, and we have empirically tested that the unstable $3$-diagonal of \cite{medina2021may_st} coincides with the unstable diagonal induced by our connected diagonal.
If $r = 5$, there are four $r$-cyclic straightenings, which yield four connected diagonals, but none of them yield the unstable diagonals obtained in \cite{medina2021may_st}; as we will see in \cref{s:suspension}, they behave differently with respect to suspension.

This new viewpoint has some advantages:
First, the construction is very explicit since it is manageable to find the coefficient of any given summand. Second, our formulas hold for the chain complexes of augmented simplicial objects too ---these do not admit an unstable diagonal in general.
Third, computing $\power^i([x])$ for small $i$'s tends to be simpler for a connected diagonal than for an unstable diagonal, while computing them for $i$'s close to $(\bars{x}-1)r$ is simpler for an unstable diagonal than for a connected diagonal.

\subsection{Other connected diagonals}

%We will see in Section \ref{s:suspension} that the connected diagonals of our main theorem are compatible with suspension.
%As a consequence, they yield also connected diagonals for the chain complex of a suspension spectra.

As we have seen, unstable diagonals and connected diagonals on the chain complex of a pointed simplicial set that are compatible with suspension yield stable diagonals in the chain complex of a suspension spectrum. This is no longer true for other models of spectra.

One of these models of spectra are the cubes in the Burnside $2$-category of Lawson, Lipshitz and Sarkar \cite{LLS20} which model Khovanov spectra \cite{LS14}. A connected $2$-cyclic diagonal for these objects was found by the first author in \cite{cantero-moran2020khovanov}, thus giving rise to explicit formulas to compute Steenrod squares in Khovanov homology. In a future work we expect to use the formulas of Theorem \ref{thm:main} to find connected $r$-cyclic diagonals for $r$ an odd prime, thus yielding explicit formulas for the odd power operations in Khovanov homology.

For the reader acquainted with the surjection operad we remark that the operations that are involved in these connected $r$-cyclic diagonals would correspond to summands indexed by non-overlapping intervals.

\subsection{Poincar\'e duality on the augmented simplex}

A key observation facilitating our construction, which we deem interesting in its own right, is the following structure on the normalized chains of the standard augmented simplices:

\begin{theorem}
	The join product defines on $\chains(\asimplex^{n})$ the structure of a Poincar\'e duality algebra with unit the empty simplex and formal dimension $n+1$.
\end{theorem}

\subsection{Outline}

In \cref{s:preliminaries} we set some sign conventions for chain complexes and the conventions for simplicial objects.
The notion of connected diagonal is presented in \cref{s:2bdiagonals}.
The construction of the connected diagonal for augmented semi-simplicial objects of Theorem \ref{thm:main} has three main ingredients.

First, an algebra structure on the chain complex of the standard augmented simplex inducing a duality theorem on normalized chains, that is used in \cref{s:3complexes} to build a convenient chain subcomplex of the tensor product $\chains(X)^{\ot r}$.
As an application of this first ingredient, we obtain a criterion to build a connected diagonal on the normalized chains of a simplicial object.

Second, in Section \ref{s:resolutions} we give a recipy to construct an explicit $\Cyc_r$-equivariant chain map $\Omega_*(r)^{\nf} \to W_*(r)$ between two resolutions of the cyclic group $\Cyc_r$. This chain map is enough to fulfil the hypotheses of the criterion.

Third, the latter recipy requires the construction of another chain map $f$, which is pursued using the barycentric and the pair subdivision of the simplex in \cref{section:atlast}. This completes the proof of Theorem \ref{thm:main}

In \cref{s:suspension} we show that these operations behave well with respect to suspensions.
In \cref{s:9Kanspectra} we briefly explain how our method yields natural connected diagonals in the chain complex of a Kan spectrum.
Finally, in \cref{s:formulas} we give an algorithmic presentation of the formulas.
These formulas are meant to be understandable after reading Sections \ref{s:preliminaries} and \ref{s:2bdiagonals}.
	% !TEX root = ../oddp.tex

\subsection*{Acknowledgments}

We thank Martin Palmer and Geoffroy Horel for insightful comments and questions.

A.M. is also grateful for the hospitality of both the Max Planck Institute for Mathematics and the Universit\'e Sorbonne Paris Nord.
	% !TEX root = ../oddp.tex

\section{Preliminaries}\label{s:preliminaries}

Let $r$ be a positive integer, and let $R$ be a commutative ring such that $\bF_r$ is a $R$-algebra.
Let $\Mod{R}$ be the category of $R$-modules and let $\Ch{R}$ be the category of chain complexes over $R$.

\subsection{Chain complexes and sign conventions}

Most of the chain complexes in this work will be finitely generated free complexes with a prescribed basis. The linear dual of a complex $\ucadenas$ is denoted $\ucocadenas$ and the linear dual of a chain map $f$ will be denoted $f^\dd$. Cochain complexes will be considered as negatively graded chain complexes. This allows to define chain maps and tensor products between chain complexes and cochain complexes.

We will denote the dual basis with the same letter as the original basis unless we want to emphasize the duality, in which case we will denote the dual of $u$ by $u^\dd$. The differential in the dual complex $(\ucocadenas,\partial^\dd)$ is
\begin{align*}
	(\partial^\dd f)(\tau) &= (-1)^{|f|+1} f(\partial \tau). 
\end{align*}

\begin{convention}
	If $A$ and $B$ are chain complexes, their tensor product $A \ot B$ has differential $\partial(a \ot b) = \partial a \ot b + (-1)^{|a|}a \ot \partial b$.
\end{convention}
There is an isomorphism $A^\dd \ot B^\dd \to (B \ot A)^{\dd}$ that sends $f \ot g$ to $h$ where $h(b \ot a) = g(b)\cdot f(a)$. There are two suspension functors that we denote by $\sus{}$ (right suspension) and $\susp{}$ (left suspension), defined on a chain complex $(\ucadenas,\partial)$ as
\begin{align*}
	\sus{} \ucadenas[i] &= \ucadenas[i-1] & \susp{} \ucadenas[i] &= \ucadenas[i-1] \\
	\sus{} \partial_i &= \partial_{i-1} & \susp{} \partial_i &= -\partial_{i-1},
\end{align*}
and on a chain map $f$ as $(\sus{}f)(a) = f(a)$ and $(\susp{} f)(a) = f(a)$. If $R[k]$ denotes the chain complex that consists of a single copy of $R$ in degree $k$, then there are isomorphisms
\begin{align*}
	\sus{k} \ucadenas &\cong \ucadenas \ot R[k] & \susp{k} \ucadenas \cong R[k] \ot \ucadenas.
\end{align*}
that send $a$ to $a \ot 1$ and to $1 \ot a$ respectively. As a consequence, there is an isomorphism
\begin{align*}
	(\sus{}\ucadenas )^\dd &\cong \susp{-1}\ucocadenas
\end{align*}
that sends a cochain $f$ to the cochain $\susp{}f$ defined as $(\susp{}f)(a) = f(\sus{}a)$.
Each notion of suspension yields a notion of chain map $A \to B$ of degree $k$: It may be a graded homomorphism $A \to B$ of degree $k$ that commutes with the differential (i.e., a chain map $A \to \sus{-k} B$), or a graded homomorphism $A \to B$ of degree $k$ that commutes with the differential if $k$ is even and anticommutes with the differential if $k$ is odd (i.e., a chain map $A \to \susp{-k}B$). We will mostly work with the second notion, but keep in mind that dualizing interchanges both notions. One can switch between them using the isomorphism
\[
\sus{k}\ucadenas \lra \susp{k}\ucadenas
\]
that sends $x$ to $(-1)^{k|x|}x$. It is therefore advisable to write a map $f \colon A \to B$ of degree $k$ with the first notion as $f \colon \sus{k}A \to B$ or $f \colon A \to \sus{-k}B$ and with the second notion as $f \colon \susp{k}A \to B$ or $f \colon A \to \susp{-k}B$.

\subsubsection{Tensor products of chain maps with the right suspension}

Given chain maps $f \colon A \to \sus{-|f|}C$ and $g \colon B \to \sus{-|g|}D$, the chain map
\[
f \ot g \colon A \ot B \to \sus{-|f|-|g|}C \ot D
\]
is defined as $(f \ot g)(a \ot b) = (-1)^{|f||b|} f(a) \ot g(b)$. Given chain maps
\begin{align*}
	A & \overset{f_1}{\lra} \sus{-|f_1|} C \overset{f_2}{\lra} \sus{-|f_1|-|f_2|} E \\
	B & \overset{g_1}{\lra} \sus{-|g_1|} D \overset{g_2}{\lra} \sus{-|g_1|-|g_2|} F \\
\end{align*}
then the chain maps
\begin{align*}
	(f_2\circ f_1) \ot (g_2\circ g_1) \colon &A \ot B \lra \sus{-|f_1|-|f_2|-|g_1|-|g_2|}E \ot F
	\\
	(f_2 \ot g_2)\circ (f_1 \ot g_1) \colon &A \ot B \lra \sus{-|f_1|-|f_2|-|g_1|-|g_2|}E \ot F
\end{align*}
are related by
\begin{align*}
	(f_2\circ f_1) \ot (g_2\circ g_1) &= (-1)^{|f_2||g_1|}(f_2 \ot g_2)\circ (f_1 \ot g_1)
\end{align*}

\subsubsection{Tensor products of chain maps with the left suspension}

Given chain maps $f \colon A \to \susp{-|f|}C$ and $g \colon B \to \susp{-|g|}D$, the chain map
\[
f \ot g \colon A \ot B \to \susp{-|f|-|g|}C \ot D
\]
is defined as $(f \ot g)(a \ot b) = (-1)^{|g||a|} f(a) \ot g(b)$. Given chain maps
\begin{align*}
	A & \overset{f_1}{\lra} \susp{-|f_1|} C \overset{f_2}{\lra} \susp{-|f_1|-|f_2|} E \\
	B & \overset{g_1}{\lra} \susp{-|g_1|} D \overset{g_2}{\lra} \susp{-|g_1|-|g_2|} F \\
\end{align*}
then the chain maps
\begin{align*}
	(f_2\circ f_1) \ot (g_2\circ g_1) \colon &A \ot B \lra \susp{-|f_1|-|f_2|-|g_1|-|g_2|}E \ot F
	\\
	(f_2 \ot g_2)\circ (f_1 \ot g_1) \colon &A \ot B \lra \susp{-|f_1|-|f_2|-|g_1|-|g_2|}E \ot F
\end{align*}
are related by
\begin{align*}
	(f_2\circ f_1) \ot (g_2\circ g_1) &= (-1)^{|f_1||g_2|}(f_2 \ot g_2)\circ (f_1 \ot g_1)
\end{align*}

\begin{notation}
	We will often use an upright symbol (for example, $\uchains$) to denote a variation of a previously built chain complex $\ucadenas$. For example, in \cref{s:unstable} we will introduce a chain complex $\W(r)$ while in \cref{s:connected} we will introduce a variation $\rW(r)$: the right suspension of the augmented version of $\W(r)$. In \cref{s:simplices} we introduce the normalized chain complex $\cadenas(X)$ while in \cref{s:alexander} we introduce a variation $\chains(X)$: the right suspension of $\cadenas(X)$.
\end{notation}

\subsection{Categories}\label{s:categories}

Let $\Cyc_r$ be the cyclic group $\langle\rho\mid \rho^r\rangle$.
\begin{definition}
	A \emph{linearized category with $r$-fold diagonals} is a monoidal category $\cC$ endowed with the following structure:
	\begin{itemize}
		\item a natural transformation between $r$-fold tensor products
		\[
		\rho \colon Y_1 \ot \overset{r}{\dots} \ot Y_r \lra Y_r \ot Y_1 \ot Y_2 \ot \dots \ot Y_{r-1}
		\]
		such that $\rho^r$ is the identity.
		\item A \emph{natural cyclic diagonal}, i.e., a natural transformation from the identity to the $r$-fold tensor product
		\[
		\Delta \colon Y \lra Y \ot \overset{r}{\dots} \ot Y
		\]
		such that $\rho\circ \Delta = \Delta$.
		\item a \emph{cyclic linearization functor}, i.e., a strong monoidal functor
		\[
		\abel \colon (\cC, \ot ) \lra (\Mod{R}, \ot )
		\]
		such that the following diagram commutes
		\[
		\xymatrix{
			L(Y_1 \ot \overset{r}{\dots} \ot Y_r) \ar[r]\ar[d]^{\rho} & \abel(Y_1) \ot \overset{r}{\dots} \ot 	\abel(Y_r)\ar[d]^{\rho} \\
			L(Y_{r} \ot \overset{r}{\dots} \ot Y_{r-1}) \ar[r] & \abel(Y_r) \ot \overset{r}{\dots} \ot \abel(Y_{r-1}) }
		\]
	\end{itemize}
\end{definition}
The main example of this work will be the symmetric monoidal category $\Setp$ of pointed sets with the smash product or its subcategory $\Set$ of sets with the Cartesian product. The functor $\abel$ takes a set to the free module on it, and a pointed set to the quotient of its free module by the base point.

\subsection{Augmented simplicial objects}\label{s:simplices}

For each integer $n \geq -1$, let $[n] = \{0,\dots,n\}$ denote the $n$th ordinal, with $[-1]$ denoting the empty ordinal.

\begin{definition}
	The \emph{augmented simplex category} $\asimplex$ is the category of finite ordinals and order-preserving maps between them. 
\end{definition}

An injective map $f \colon [n] \to [m]$ is called a \emph{face map}, while a surjective map is called a \emph{degeneracy}. A face map $f$ is determined by the subset $U = [m]\smallsetminus f([n])$, and is denoted $d^U$. If $U$ has a single element $i$, we will write $d^i$ instead of $d^{\{i\}}$. There are $n+1$ surjections $s^i \colon [n+1] \to [n]$, $i = 0\dots n$ characterized by $s^i(i) = s^i(i+1)$. Every map in this category factors as an injection followed by a surjection.

The category $\asimplex$ has a strict monoidal product called the \emph{join}, that sends a pair of ordinals $[n],[m]$ to the ordinal $[n+m+1]$ and a pair of functions $f \colon [n] \to [q]$, $g \colon [m] \to [p]$ to the function $f*g \colon [n+m+1] \to [q+p+1]$ defined as $f(x) = x$ if $x\leq n$ and $f(x) = q+1+f(x-n)$ if $x \geq n$. The unit is the empty ordinal $[-1]$.

\begin{definition}
	An \emph{augmented simplicial object in $\cC$} is a functor $X \colon \asimplex^\mathrm{op} \to \cC$.
	An \emph{augmented cosimplicial object in $\cC$} is a functor $X \colon \asimplex \to \cC$.
\end{definition}

If $X$ is an augmented simplicial object, $X([n])$ is denoted $X_n$ and the morphisms $X(d^U)$ and $X(s^i)$ are denoted $\d_U$ and $s_i$ respectively. If $X$ is an augmented cosimplicial object, $X([n])$ is denoted $X_n$ and the morphisms $X(d^U)$ and $X(s^i)$ are denoted $d^U$ and $s^i$ respectively.

The inclusion $\simplex \subset \asimplex$ of the non-empty ordinals induces a restriction functor $\cC^{\asimplex^{\op}} \to \cC^{\simplex^{\op}}$ to the category of simplicial objects on $\cC$ that forgets $X_{-1}$. It has a left adjoint that sets $X_{-1}$ to be the colimit of $X$ and a right adjoint that sets $X_{-1}$ to be the terminal object of $\cC$.

If $X$ is an augmented simplicial object in one of the categories of Section \ref{s:categories}, postcomposing with the linearization functor $\abel \colon \cC \to \Mod{R}$ yields a simplicial $R$-module $\abel(X)$.

Define the \emph{complex of unnormalized chains $\ucadenas(X;R)$ with coefficients in $R$} as
\begin{align*}
	\ucadenas[n](X;R) &= \abel(X)_n
	&
	\partial(x) &= \sum_{j=0}^{n} (-1)^jd_j(x)
\end{align*}

If $X$ is an augmented simplicial object, the \emph{complex of normalized chains $\cadenas(X;R)$ with coefficients in $R$} is defined as
\begin{align*}
	\cadenas[n](X;R) &= L'(X)_n
	&
	\partial(x) &= \sum_{j=0}^{n} (-1)^jd_j(x)
\end{align*}
where $L'(X)_n$ is obtained from $L(X)_n$ by quotienting the image of the degeneracies.

\begin{notation}\label{notation:chains}
	We will denote by $\chains(X)$ the right suspension of $\cadenas(X)$.
\end{notation}

\begin{notation}
	A few times we will be working with normalized chain complexes of simplicial objects that they themselves are part of a cosimplicial or simplicial object. We will denote the ``internal'' face maps that are used to build the differential as $\diff_i$ and the external face maps that yield the cosimplicial or simplicial structure as $\d_i$.
\end{notation}

The \emph{standard augmented $n$-simplex} is the augmented simplicial set $\simplex^n_+$ given by $[m] \to \mathrm{Mor}_{\asimplex}([m],[n])$. The normalized chain complex $\cadenas(\asimplex^n)$ is generated by the non-degenerate simplices, which are indexed by the inclusions $\d_U \colon [m] \to [n]$. The simplex indexed by $\d_U$ is denoted $[v_0,\dots,v_{m}]$, where $\{v_0,\dots,v_{m}\} \subset \{0,\dots,n\}$ is the ordered image of $[m]$ by $\d_U$. As usual, if $\perm$ is a permutation of $0,\dots,m$, then $[v_{\perm(0)},\dots,v_{\perm(m)}]$ will denote the generator $(-1)^{\sign{\perm}}[v_0,\dots,v_m]$.

\subsection{The Milgram resolution of a group \texorpdfstring{$G$}{G}}

Recall the that the Milgram construction applied to a group $G$, yields a contractible augmented simplicial set $\EG$ with a free $G$-action:
\begin{align*}
	(\EG)_k &= G^{k+1}\\
	\d_i(g_0\dots g_k) &= (g_0\dots\hat{g}_i\dots g_k) \\
	\s_i(g_0\dots g_k) &= (g_0\dots g_i,g_i\dots g_k)
\end{align*}

	% !TEX root = ../oddp.tex

\section{Cyclic diagonals and power operations}\label{s:2bdiagonals}

Let $\rho$ denote the standard generator of the cyclic group $\Cyc_r$ on $r$ elements, which we identify with $\{0,1,\dots,r-1\}$ via $i\mapsto \rho^i$.
The symmetric group on $\{0,1,\dots,r-1\}$ is denoted $\Sym_r$, and the regular representation includes $\Cyc_r$ into $\Sym_r$.
The letter $\tilde{r}$ will denote the quantity $\frac{r-1}{2}$.
The definitions of \cref{s:unstable} are taken from \cite{may1970general} and \cite{medina2021may_st}.
Chain complexes in this section are assumed to be non-negatively graded.

\subsection{Unstable diagonals}\label{s:unstable}

The \emph{minimal $\Cyc_r$-resolution} of the commutative ring $R$ with trivial action is the chain complex $\W(r) = \bigoplus_{q \geq 0} \W[q](r)$ defined as follows:
\begin{align*}
	\W[q](r) &= R[\Cyc_r]\langle e_q\rangle &
	\partial(e_q) &= \begin{cases}
		N(e_{q-1}) & \text{if $q$ is even} \\
		T(e_{q-1}) & \text{if $q$ is odd,}
	\end{cases}
\end{align*}
where $N = \sum_j \rho^j$ and $T = \rho - \Id$.

\begin{definition}
	An \emph{unstable $r$-cyclic diagonal} on a chain complex $\ucadenas$, or \emph{unstable $r$-diagonal} for short, is a $\Cyc_r$-equivariant chain map
	\[
	\mu \colon \W(r) \ot \ucadenas \lra \ucadenas^{\ot r}.
	\]
\end{definition}

\begin{proposition}\label{prop:unstable}
	If $r$ is prime, an unstable diagonal $\mu$ on a chain complex $\ucadenas$ induces operations indexed by the integers
	\begin{align*}
		\power^{i} \colon H_{-m}(\ucocadenas;\mathbb{F}_r)& \lra H_{-m-i}(\ucocadenas;\mathbb{F}_r)
	\end{align*}
	that send a cohomology class $[x]$ in degree $-m$ to the cohomology class $[y]$ with
	\[
	y(a) = (-1)^{\binom{m}{2} \tilde{r}}(\tilde{r}!)^m(x \ot \overset{r}{\dots} \ot x)(\mu(e_{(r-1)m-i} \ot a)).
	\]
	Furthermore, these operations satisfy the following equation
	\begin{itemize}
		\item $\power^i(x) = 0$ if $i>(r-1)m$ and $x$ has degree $-m$.
	\end{itemize}
\end{proposition}

\begin{proof}
	See \cite{may1970general}. For the last claim, note that if $i>(r-1)m$, then $(r-1)m-i$ is negative, and therefore $\mu(e_{(r-1)m-i},a)$ vanishes.
\end{proof}

\begin{definition}
	An unstable $r$-diagonal $\mu$ on a chain complex $\ucadenas$ is said to be \emph{May--Steenrod} if $\mu$ factors as
	\[
	\W(r) \ot \ucadenas \xra{f \ot \id} V_*(r) \ot \ucadenas \lra \ucadenas^{\ot r}
	\]
	where $V_*(r)$ is a free resolution of the symmetric group on $r$ letters, $f$ is equivariant with respect to the regular representation of $\Cyc_r$, and the second map is $\Sym_r$-equivariant.
\end{definition}

\begin{remark}
	If $\mu$ is May--Steenrod, then $\power^i = 0$ unless $i = 2(r-1)k$ or $i = 2(r-1)k+1$.
\end{remark}

\begin{example}
	The cup-$i$ coproducts of Steenrod \cite{steenrod1947products, medina2022axiomatic} yield a natural unstable $2$-diagonal on the normalized chains of simplicial sets which is May--Steenrod.
	For general $r$, a similar May--Steenrod unstable $r$-diagonal was constructed in \cite{medina2021may_st}.
\end{example}

\subsection{Stable diagonals}\label{s:stable}

Let $\Wst(r)$ be the unbounded chain complex
\[
\dotsb \lra \Cyc_r\langle e_{-2}\rangle
\overset{T}{\lla}
\Cyc_r\langle e_{-1}\rangle
\overset{N}{\lla}
\Cyc_r\langle e_0\rangle
\overset{T}{\lla}
\Cyc_r\langle e_{1}\rangle
\overset{N}{\lla}
\Cyc_r\langle e_{2}\rangle
\overset{T}{\lla}
\dots
\]
If either $r$ is odd or $r=2$ and $R=\bF_2$, we understand it as the limit of
\[
\dotsb \lra \sus{2(1-r)}\W(r) \xra{\theta_{1-r}} \sus{1-r} \W(r) \xra{\theta_{1-r}} \W(r).
\]
In general, we understand it as the limit of
\[
\dotsb \lra \sus{4(1-r)}\W(r)\overset{\theta_{2(1-r)}}{\lra} \sus{2(1-r)} \W(r) \overset{\theta_{2(1-r)}}{\lra} \W(r).
\]
where $\theta_{k} \colon \W(r) \to \sus{-k} \W(r)$ is the \emph{suspension homomorphism} that sends $e_{q}$ to $e_{q+k}$, which only exists if $k$ is even and non-positive. This complex is isomorphic to the desuspension of the complex that computes the Tate homology of $\Cyc_r$, which is part of a homotopy fibre sequence
\[
\Wst(r) \lra \W(r)\overset{N}{\lra} \Wd(r)
\]
where the second map is the norm map that sends $e_q$ to zero unless $q=0$, in which case it is sent to $N(e_0^\dd)$.
Notice that $\Wst(r)$ is not equal to the standard model of the homotopy fibre: in the homotopy fibre we have, for negative $i$, that $\partial(e_i) = -N(e_{i-1})$ or $T(e_{i-1})$ depending on whether $i$ is even or odd.

\begin{definition}
	A \emph{stable $r$-cyclic diagonal} on a chain complex $\ucadenas$, or simply a \textit{stable $r$-diagonal}, is a $\Cyc_r$-equivariant chain map
	\[
	\mu \colon \Wst(r) \ot \ucadenas \lra \ucadenas^{\ot r}.
	\]
\end{definition}

\begin{definition}
	Let $\ESst_r$ be the inverse limit of
	\[
	\dotsb \lra \sus{2(1-r)} \ES_r\overset{\theta_{r-1}}{\lra}\sus{1-r} \ES_r\overset{\theta_{r-1}}{\lra} \ES_r
	\]
	where $\theta_{r-1}$ is the operadic suspension of \cite{berger2004combinatorial} that sends a sequence of permutations $(\sigma_0,\dots,\sigma_q)$ to zero if $(\sigma_0(0),\dots,\sigma_{r-1}(0))$ has repeated elements, and to $\pm (\sigma_{r-1},\sigma_r,\dots,\sigma_q)$ otherwise, where $\pm$ is the sign of the permutation that orders the sequence $(\sigma_0(0),\dots,\sigma_{r-1}(0))$.
\end{definition}

\begin{definition}
	A stable $r$-diagonal $\mu$ on a chain complex $\ucadenas$ is said to be $\emph{May--Steenrod}$ if it factors as
	\[
	\mu \colon \Wst(r) \ot \ucadenas \xra{f \ot \id} \ESst_r \ot \ucadenas \lra \ucadenas^{\ot r}
	\]
	with $f$ equivariant with respect to the inclusion $\Cyc_r \subset \Sym_r$ and the second map is $\Sym_r$-equivariant.
\end{definition}

Every unstable diagonal yields a stable diagonal by precomposing $\mu$ with the projection $\Wst(r) \to \W(r)$.
A stable diagonal comes from an unstable diagonal if it vanishes on $\Wst[i](r) \ot \ucadenas$ for $i<0$.

\begin{proposition}
	If $r$ is prime, a stable $r$-cyclic diagonal $\mu$ on a non-negatively graded chain complex $\ucadenas$ induces operations indexed by the integers
	\begin{align*}
		\power^{i} \colon H_{-n}(\ucocadenas;\mathbb{F}_r)& \lra H_{-n-i}(\ucocadenas;\mathbb{F}_r)
	\end{align*}
	that send a cohomology class $[x]$ in degree $-m$ to the cohomology class $[y]$ with
	\[
	y(a) = (-1)^{\binom{m}{2} \tilde{r}}(\tilde{r}!)^m(x \ot \overset{r}{\dots} \ot x)(\mu(e_{(r-1)m-i} \ot a))
	\]
\end{proposition}

\begin{proof}
	If $\ucadenas$ has a stable diagonal, then $\sus{} \ucadenas$ has a stable diagonal as well, defined as $\mu(e_q \ot \sus{} a) = \sus{\ot r}\mu(e_{q+r-1} \ot a)$.
	Therefore, after enough suspensions, we may assume that the stable diagonal vanishes in negative degrees and that $q$ is positive.
	Then a truncation reduces this to the situation of an unstable diagonal.
\end{proof}

\begin{example}
	As we will see in \cref{s:suspension}, the unstable diagonal for normalized chains in simplicial sets of the previous example, defines a stable diagonal on the normalized chains of simplicial spectra (these are defined in \cite{Gill2020}) which is May--Steenrod.
\end{example}

\begin{figure}
	\input{tables/unstable}
	\caption{A depiction of $\Wst(3) \ot \ucadenas$. A stable diagonal comes from an unstable diagonal if it vanishes on the part coloured in green.}
\end{figure}

\subsection{Connected diagonals}\label{s:connected}

Let $\mu$ be a stable diagonal on $\ucadenas$.
A way of guaranteeing the desirable property that $\power^i(x)$ vanishes for $i<0$ is to ask for $\mu$ to vanish on $\Wst[i] \ot \ucadenas[n]$ for $i>(r-1)n$.
We will now introduce a convenient chain complex $\rWd(r) \hotimes \rucadenas$ such that giving a chain map $\mu \colon \rWd(r)  \hotimes \rucadenas \to \ucadenas^{\ot r}$ is equivalent to give a stable diagonal that vanishes on $\Wst[i](r) \ot \ucadenas[n]$ for $i>(r-1)n$ (see Figures \ref{figure:2} and \ref{figure:3}).

\begin{figure}
	\input{tables/hat}
	\caption{If a stable diagonal for $r=3$ vanishes on the part coloured in green, then $\power^i$ vanishes for $i<0$.}
	\label{figure:2}

 \vspace{2cm}

	\input{tables/hat_alone}
	\caption{The complex $\rWd(r) \hotimes \rucadenas$ for $r=3$ and $R=\bF_3$.
		The map $\varphi$ goes from Figure \ref{figure:2} to here, is an epimorphism and its kernel is the smallest chain subcomplex that contains the green part. The power operation $\power^0$ is computed at the corners $\bF_3\hotimes C^3_n$.}
	\label{figure:3}
\end{figure}

Let $\rucadenas$ be the chain complex $\bigoplus_m \sus{(r-1)m}\ucadenas[m]$ with the same differential as $\ucadenas$, which now has degree $-r$.
An element $x\in \ucadenas$ viewed in $\rucadenas$ will be denoted $\vec{x}$.
We write $\rW(r)$ for the right suspension of the augmented chain complex $\W(r) \to R$, i.e., the chain complex
\[
R
\overset{N}{\lla}
R[\Cyc_r]\langle e_{1}\rangle
\overset{T}{\lla}
R[\Cyc_r]\langle e_{2}\rangle
\overset{N}{\lla}
R[\Cyc_r]\langle e_{3}\rangle
\overset{T}{\lla}
\dotsb
\]
(note the reindexing, so that $e_i$ has degree $i$) and $\rWd(r)$ for its dual:
\[
\dotsb
\overset{T}{\lla}
R[\Cyc_r]\langle e_{-3}^\dd\rangle
\overset{-N}{\lla}
R[\Cyc_r]\langle e_{-2}^\dd\rangle
\overset{T}{\lla}
R[\Cyc_r]\langle e_{-1}^\dd\rangle
\overset{-N}{\lla}
R.
\]
Observe that $\rho$ acts ``backwards'' in the dual complex, in the sense that $\rho (e_{-i}^\dd) = (\rho^{-1} e_i)^{\dd}$.
The generator of the copy of $R$ in the right will be denoted as $1$ or as $e_0^\dd$.

Again, there is a suspension map $\theta_k \colon \rWd(r) \to \susp{-k}\rWd(r)$ defined for all even $k$ (or all $k$ in case $r=2$ and $R=\bF_2$) as
\begin{align*}
	\theta_k(e_{-i}^\dd) &= e_{k-i}^\dd
	&
	\theta_k(1) &= \begin{cases} N(e_k^\dd) &\text{if $k<0$} \\ 1 &\text{if $k=0$} \\ 0 &\text{if $k>0$.} \end{cases}
\end{align*}
If either $r$ is odd or $r=2$ and $R=\mathbb{F}_2$, define the chain complex $\rWd(r) \hotimes \rucadenas$ as the graded $R[\Cyc_r]$-module $\rWd(r) \ot \rucadenas$ with the differential
\[
\partial(e^\dd_{-q} \ot \vec{x}) = \partial (e^\dd_{-q}) \ot \vec{x} + (-1)^q \theta_{r-1}(e^\dd_{-q}) \ot \partial (\vec{x}).
\]
If $r$ is even and $R$ is arbitrary, define the cochain complex $\rWhatd(r)$ as
\[
\dotsb
\overset{-T}{\lla} 
R[\Cyc_r]\langle e_{-2}^\dd\rangle 
\overset{N}{\lla}
R[\Cyc_r]\langle e_{-1}^\dd\rangle
\overset{-T}{\lla} 
R[\Cyc_r]\langle e_{0}^\dd\rangle / N 
\]
This chain complex admits chain maps
\begin{align*}
	\hat{\theta}_{r-1} \colon \rWd(r)& \lra \susp{1-r}\rWhat(r)
	&
	\hat{\theta}_{r-1} \colon \rWhat(r)& \lra \susp{1-r}\rW(r).
\end{align*}
Define $\ucadenas[\mathrm{even}] = \bigoplus_{m}\ucadenas[2m]$ and $\ucadenas[\mathrm{odd}] = \bigoplus_{m}\ucadenas[2m+1]$ and write $\rWd(r) \hotimes \rucadenas$ for the graded $R[\Cyc_r]$-module $\left(\rWd(r) \ot \rucadenas[\mathrm{odd}]\right)\oplus \left(\rWhatd(r) \ot \rucadenas[\mathrm{even}]\right)$ with differential
\[
\partial(e^\vee_{-q} \ot \vec{x}) = \partial (e^\vee_{-q}) \ot \vec{x} + (-1)^q \theta_{1-r}(e^\vee_{-q}) \ot \partial (\vec{x}).
\]
Observe that if $r=2$ and $R=\bF_2$, the complexes $\rWd(r)$ and $\rWhatd(r)$ coincide, and so do both definitions of $\rWd(r) \hotimes \rucadenas$.

\begin{definition}
	A \emph{connected $r$-cyclic diagonal} on a non-negatively graded chain complex $\ucadenas$, or a \textit{connected $r$-diagonal} for short, is a $\Cyc_r$-equivariant chain map
	\[
	\mu \colon \rWd(r) \hotimes \rucadenas \lra \ucadenas^{\ot r}.
	\]
\end{definition}

If either $k$ is even or $r=2$ and $R=\bF_2$, let $\varphi_k \colon \Wst(r) \to \susp{-k}\rWd(r)$ be the chain map that sends $e_q$ to $(-1)^{\binom{k-q}{2}}e_{k-q}^{\vee}$.
One can obtain this chain map as follows: We already saw $\Wst(r)$ as the homotopy fibre of the norm map $N \colon \W(r) \to \Wd(r)$, but it is also the homotopy fibre of the map $\rWd(r) \to \sus{}\rW(r)$ that is the identity on $R$ and zero elsewhere.
The map $\varphi_k$ is the composition
\[
\Wst(r)\overset{\theta_k}{\lra} \sus{-k} \Wst(r) \overset{\cong}{\lra} \sus{-k}\hofib(\rWd(r) \to \rW(r)) \lra \susp{-k}\rWd(r).
\]
This composition is an epimorphism with kernel the smallest chain subcomplex that contains $\Wst[>k]$.
In other words, it is isomorphic to the truncation $\tau_k \Wst(r)$.
If $k$ is odd, let $\varphi_k \colon \Wst(r) \to \sus{k}\rWhatd(r)$ be the chain map that sends $e_q$ to $e_{k-q}^{\vee}$.

Define the chain map $\varphi \colon \Wst(r) \to \rWd(r) \hotimes \rucadenas$ as
\[
\varphi(e_q \ot x) = \varphi_{-n(r-1)}(e_{q}) \ot \vec{x}
\]
if $x$ has degree $n$.
This is an epimorphism with kernel the smallest chain subcomplex that contains $\bigoplus_n\bigoplus_{i>(r-1)n} \Wst[i] \ot \ucadenas[n]$.
As a consequence,

\begin{proposition}
	If $r$ is prime, a connected $r$-cyclic diagonal induces operations indexed by the integers
	\begin{align*}
		\power^{i} \colon H_{-m}(\ucocadenas;\mathbb{F}_r)& \lra H_{-m-i}(\ucocadenas;\mathbb{F}_r)
	\end{align*}
	that send a cohomology class $[x]$ in degree $m$ to the cohomology class $[y]$ with
	\[
	y(a) = (-1)^{\binom{ir}{2} + \binom{m}{2}\tilde{r}}(\tilde{r}!)^m(x \ot \overset{r}{\dots} \ot x)(\mu(e_{ri}^\vee \ot a))
	\]
	These operations satisfy the following equation
	\begin{itemize}
		\item $\power^i(x) = 0$ for $i<0$.
	\end{itemize}
\end{proposition}

\begin{example}
	As far as we know, the concept of connected diagonal has not appeared before in the literature, nonetheless, this viewpoint is implicit in the construction of the unstable diagonal for $r=2$ on the cochain operations on the normalized cochains of a simplicial set of \cite{medina2021fast_sq}.
	It is also implicit in the construction of the stable diagonal for $r=2$ of the cochain operations on the cochains of augmented semi-simplicial objects in the Burnside $2$-category \cite{cantero-moran2020khovanov}.
\end{example}

\begin{question}
	Give a good definition of ``May--Steenrod connected $r$-diagonal''.
	Give a homotopy equivalence between the Tate complex of $\ES_r$ and the unbounded complex $\ESst_r$.
\end{question}
	% !TEX root = ../oddp.tex

\section{Alexander duality on the simplex and tensor products}\label{s:3complexes}

In this section we present a duality on the normalized chains of the augmented simplex, and use it to build manageable models of the $r$-fold tensor product of the normalized chain complex of an augmented simplicial object in $\cC$. In Proposition \ref{prop:omegarm} and Lemma \ref{lemma:omegar} we reduce the construction of a connected diagonal to the construction of a map between two resolutions of the cyclic group $\Cyc_r$ satisfying an additional hypothesis. This latter map is constructed in \cref{s:resolutions}.

\subsection{Alexander duality on the augmented standard simplex}\label{s:alexander}

Given a simplicial set $X$, let $\chains(X) = \sus{}\cadenas(X)$ be the right suspension of the normalized chains of $X$.

\begin{definition}\label{d:poincare_duality_algebra}
	Let $A$ be a connected commutative algebra that is finite dimensional for each degree.
	We say $A$ is a \textit{Poincar\'e duality algebra} of \textit{formal dimension} $d$ if:
	\begin{enumerate}
		\item\label{i:pd1} $A_i = 0$ for $i > d$,
		\item\label{i:pd2} $\dim A_d = 1$,
		\item\label{i:pd3} $A_i \ot A_{d-i} \to A_d$ is non-degenerate.
	\end{enumerate}
\end{definition}

\begin{definition}\label{d:join_product}
	For any $[n] \in \ob\asimplex$ the \textit{join product} $\ast \colon \chains(\asimplex^{n})^{\ot 2} \to \chains(\asimplex^{n})$ is the linear map defined by sending a basis element $[v_0, \dots, v_{p-1}] \ot [v_{p},\dots,v_{m-1}]$ to $(-1)^{\sign{\perm}}[v_{\perm(0)}, \dots, v_{\perm(m-1)}]$	if $v_i \neq v_j$ for all $i \neq j$, where $\perm$ is the permutation ordering the vertices, and to $0$ otherwise.
\end{definition}

\begin{notation}
	If $\tau = [v_0,\dots,v_{k-1}]$ and $\tau'=[u_0,\dots,u_{m}-1]$ are two non-degenerate simplices of $\asimplex^n$, let $\lambda(\tau,\tau')$ be the parity of the permutation that orders the sequence $[v_0,\dots,v_{k-1},u_0,\dots,u_{m-1}]$.
\end{notation}

\begin{theorem}
	The join product is a chain map naturally defining on each $\chains(\asimplex^{n})$ the structure of a Poincar\'e duality algebra with unit the empty simplex $[-1] \to [n]$ and formal dimension $n+1$.
\end{theorem}

\begin{proof}
	The complex $\chains(\asimplex^{n})$ is connected and satisfies \cref{i:pd2,i:pd3} in \cref{d:poincare_duality_algebra} since $\chains(\asimplex^{n})_0 \cong \Z\{[-1] \to [n]\}$, $\chains(\asimplex^{n})_{n+1} \cong \Z\{[n] \to [n]\}$, and $\chains(\asimplex^{n})_{n+k} \cong 0$ for $k>1$.

	That the join product is a natural chain map can be easily verified and a complete proof is presented in \cite[p.19]{medina2020prop1}.

	Thinking about the join product in terms of the union of sets with a permutation sign leads to a direct verification of its commutativity (in the graded sense) and unitality with respect to the empty simplex.

	To verify \cref{i:pd3} consider a basis element $x = [v_1,\dots,v_i]$.
	Let $x^c$ be the ordered complement of $\set{v_1,\dots,v_i}$ in $\{0,\dots,n\}$ and notice that $x \ast x^c = \pm [0,\dots,n]$ as required.
\end{proof}

This pairing thus induces an isomorphism
\begin{equation}\label{eq:iso1}
	\Lambda \colon \chains(\asimplex^n) \to \susp{n+1}\cochains(\asimplex^n)
\end{equation}
that sends a chain $\tau$ to $(-1)^{\lambda(\tau,\tau^c)}(\tau^c)^\vee$. We will refer to this isomorphism as \emph{Alexander duality}.

\begin{remark}\label{remark:alex}
	If $\tau = [v_0,\dots,v_k]$ and $\tau^c = [u_0,\dots,u_m]$ are the ordered representatives, then
	\begin{align*}
		\lambda(\tau,\tau^c)&\equiv \sum_{i=0}^k (v_i-i) \equiv \sum_{i=0}^k v_i - \binom{k+1}{2} \\
		\lambda(\tau,\tau^c)&\equiv \sum_{j=0}^m (n-u_j-m+j) \equiv \sum_{j=0}^m u_j +(n-m)(m+1)+\binom{m+1}{2}.
	\end{align*}
\end{remark}

\begin{notation}
	We will denote normalized chains between brackets and normalized cochains between parentheses.
\end{notation}
	% !TEX root = ../oddp.tex

\subsection{Simplicial functor tensor products}

Recall that if $A$ is an augmented cosimplicial chain complex and $B$ is an augmented simplicial chain complex, the functor tensor product $A \ot_{\asimplex} B$ is the quotient of the tensor product $A \ot B$ by the relation generated by
\begin{align*}
	\d^i a \ot b &\sim (-1)^{|\d^i||a|}a \ot \d_i b
	&
	\s^ia \ot b &\sim (-1)^{|\s^i||a|}a \ot \s_ib.
\end{align*}
The chain complexes $\chains(\asimplex^\bullet)$ form an augmented cosimplicial chain complex with
\begin{align*}
	\d^i \colon \chains(\asimplex^n)& \lra \chains(\asimplex^{n+1})
	&
	\s^i \colon \chains(\asimplex^{n})& \lra \chains(\asimplex^{n-1})
\end{align*}
given on an ordered representative $[v_0,\dots,v_{k-1}]\in \chains(\asimplex^n)$ by
\begin{align*}
\d^i([v_0,\dots,v_{k-1}]) &= [v_0,\dots,v_{j-1},v_j+1,\dots,v_{k-1}] \quad \text{if $v_{j-1}<i\leq v_j$}
\\
\s^i([v_0,\dots,v_{k-1}]) &= [v_0,\dots,v_{j-1},v_j-1,\dots,v_{k-1}-1] \quad \text{if $v_{j-1} \leq i < v_j$}.
\end{align*}
If $X$ is an augmented simplicial object in $\cC$, the chain complex $\chains(X)$ can alternatively be described as the functor tensor product over $\asimplex$ of the augmented cosimplicial chain complex $\chains(\asimplex^{\bullet})$ and the augmented simplicial $R$-module $\abel(X)_\bullet$:
\begin{equation}\label{eq:1}
	\chains(\asimplex^\bullet) \ot_{\asimplex} \abel(X)_\bullet \cong \chains(X).
\end{equation}
This sends a simplex $[k] \to [n]$ and a simplex $[n] \to X$ to the composition $[k] \to X$.

Via Alexander duality, we may endow $\cochains(\asimplex^n)$ with the following augmented \emph{cosimplicial} structure
\begin{align*}
	\d^i \colon \cochains(\asimplex^n) & \lra \susp{}\cochains(\asimplex^{n+1})
	&
	\s^i \colon \cochains(\asimplex^{n+1}) & \lra \susp{}\cochains(\asimplex^n)
\end{align*}
given on a generator $U=(u_0,\dots,u_{m-1})$ by
\begin{align*}
	\d^i(U) &=
	(-1)^{i+j+m+n+1}(u_0,\dots,u_{j-1},i,u_{j}+1,\dots,u_{m-1}+1)
	\quad \text{if $u_{j-1}<i \leq u_{j}$}
	\\
	\s^i(U) &= \begin{cases}
	(-1)^{i+j+m+n+1}(u_0,\dots,u_{j-1},u_{j+1}-1\dots,u_{m-1}-1) & \text{if $u_{j} = i$,}
	\\
	0 & \text{otherwise.}
	\end{cases}
\end{align*}
 Let $\rA(X)$ be the augmented simplicial graded $R$-module with $\rA(X)_n = \susp{n+1}\abel(X)_{n}$, and face maps and degeneracies are induced by those of $\abel(X)$ and have degree $-1$ and $+1$ respectively.

\begin{lemma}\label{lemma:2} The map
	\[
	\cochains(\asimplex^\bullet) \ot_{\asimplex} \rA(X)_\bullet \lra \chains(\asimplex^\bullet) \ot_{\asimplex} \abel(X)_\bullet
	\]
	that sends $U \ot \tau$ with to $\Lambda^{-1}(U) \ot \tau$ is an isomorphism of chain complexes.
\end{lemma}
More specifically, the map sends $U \ot \tau$ to $(-1)^{\lambda(U^c,U)+(n+1)m} U^c \ot \tau$.
\begin{example}\label{example:first3} If $U \ot \tau\in \cochains(\asimplex^\bullet) \ot_{\asimplex} \rA(X)_\bullet$ is given, one can compute its image on $\chains(X)$ using the relations
\[
	\d^i U \ot \tau = (-1)^{m} U \ot \d_i\tau
\]
 (with $U$ of degree $m$) in the functor tensor product and that $\emptyset \ot \tau \mapsto \tau$. For example, if $r=3$ and $X=\asimplex^7$ and $\tau = [0,1,2,3,4,5,6,7]$ is the top simplex,
\begin{align*}
		(0,6) \ot \tau &= (-1)^{6+1+2+8}(-1)^2(0) \ot [0,1,2,3,4,5,7]
		\\
		&= -(-1)^{0+0+1+8}(-1)^1\emptyset \ot [1,2,3,4,5,7]\mapsto -[1,2,3,4,5,7].
\end{align*}
Alternatively, using the map of Lemma \ref{lemma:2},
\[
	(0,1,3,4) \ot \tau = +(-1)^{8\cdot 4}[2,5,6,7].
\]
\end{example}

\begin{example}\label{example:first3'}
	If $r=3$ and $X=\asimplex^{19}$ and $\tau = [0,2,3,4,5,6,9]$,
	\begin{align*}
		((0,6) \ot \tau) &= -(-1)^{20\cdot 2} [2,3,4,5,6]
		\\
		((0,1,3,4) \ot \tau) &= +(-1)^{20\cdot 4}[3,6,9].
	\end{align*}
\end{example}

\begin{example}\label{example:first5}
	If $r=5$ and $X=\asimplex^2$ and $\tau = [0,1,2]$ is the top simplex,
	\[
	\begin{split}
	(1) \ot \tau &= - (-1)^{3\cdot 1} [0,2]
	\\
	(1,2) \ot \tau &= +(-1)^{3\cdot 2} [0]
	\\
	(0,1,2) \ot \tau &= +(-1)^{3\cdot 3}\emptyset
	\\
	(2) \ot \tau &= +(-1)^{3\cdot 1} [0,1].
	\end{split}
	\]
\end{example}

The $r$-fold tensor product $\cochains(\asimplex^\bullet)^{\ot r}$ becomes an augmented cosimplicial chain complex with face maps of degree $-r$, while the $r$-fold tensor product $\rA(X)_\bullet^{\ot r}$ becomes a simplicial chain complex with face maps of degree $-r$. Specifically, the faces and degeneracies of the former are as follows:
\begin{align*}
	\d^i \colon \cochains(\asimplex^n)^{\ot r}& \lra \susp{r}\cochains(\asimplex^{n+1})^{\ot r}
	&
	\s^i \colon \cochains(\asimplex^{n+1})^{\ot r}& \lra \susp{-r}\cochains(\asimplex^{n})^{\ot r}
\end{align*}
given by
\begin{align*}
	\d^i(U_0 \ot \dots \ot U_{r-1}) &= (-1)^{\sum_k (r-1-k)m_k}\d^i(U_0) \ot \dots \d^i(U_{r-1})
	\\
	\s^i(U_0 \ot \dots \ot U_{r-1}) &= (-1)^{\sum_k (r-1-k)m_k}\d^i(U_0) \ot \dots \d^i(U_{r-1})
\end{align*}
where $k = 0\dots r-q$ and $m_k$ is the degree of $U_k$ (the sign comes from the fact that the maps $\d^i$ in each tensor factor have degree $1$, hence one has to pay a sign for moving it accross the other tensor factors). Similarly, the $r$-fold tensor product $\rA(X)^{\ot r}_{\bullet}$ becomes an augmented simplicial graded $R$-module with face maps of degree $-r$ and degeneracy maps of degree $r$
\begin{align*}
	\d_i \colon \rA(X)_n^{\ot r} & \lra \rA(X)_{n-1}^{\ot r}
	&
	\s_i \colon \rA(X)_n^{\ot r} & \lra \rA(X)_{n+1}^{\ot r}
\end{align*}
given by
\begin{align*}
	\d_i(\tau_0 \ot \dots \ot \tau_{r-1}) &= (-1)^{\sum_k (r-1-k)(n+1)} \d_i(\tau_0) \ot \dots \ot \d_i(\tau_{r-1})
	\\
	\s_i(\tau_0 \ot \dots \ot \tau_{r-1}) &= (-1)^{\sum_k (r-1-k)(n+1)} \s_i(\tau_0) \ot \dots \ot \s_i(\tau_{r-1})
	\end{align*}
\begin{lemma}\label{lemma:3}
	The chain map
	\[\cochains(\asimplex^\bullet)^{\ot r} \ot_{\asimplex} \rA(X)_\bullet^{\ot r} \lra \left(\cochains(\asimplex^\bullet) \ot_{\asimplex} \rA(X)_\bullet\right)^{\ot r}\]
	given by
	\begin{align*}
		(U_0 \ot \dots \ot U_{r-1}) \ot &(\tau_0 \ot \dots \ot \tau_{r-1})
	\\
 &\mapsto (-1)^{(n+1)\sum_{k} km_k}(U_0 \ot \tau_0) \ot \dots \ot (U_{r-1} \ot \tau_{r-1})
	\end{align*}
where $k=0\dots r-1$, the element $\tau$ has degree $n+1$ and $U_k$ has degree $m_k$, is a $\Cyc_r$-equivariant monomorphism of chain complexes.
\end{lemma}
\begin{remark}
	Had we worked with augmented semi-simplicial objects, this map would be only a monomorphism.
\end{remark}
\begin{example}\label{example:reord} If $r=3$ and $X=\asimplex^7$ and $\tau = [0,1,2,3,4,5,6,7]$ is the top simplex,
\[
		\alpha((\emptyset \ot (0,6) \ot (0,1,3,4)) \ot \tau^{\ot 3}) = (-1)^{8\cdot 10}(\emptyset \ot \tau) \ot ((0,6) \ot \tau) \ot ((0,1,3,4) \ot \tau)
	\]
 If $r=3$ and $X=\asimplex^{19}$ and $\tau = [0,2,3,4,5,6,9]$,
\[
		\alpha((\emptyset \ot (0,6) \ot (0,1,3,4)) \ot \tau^{\ot 3}) = (-1)^{7\cdot 10}(\emptyset \ot \tau) \ot ((0,6) \ot \tau) \ot ((0,1,3,4) \ot \tau)
	\]
 If $r=5$ and $X=\asimplex^2$ and $\tau = [0,1,2]$ is the top simplex,
	\[
	\begin{split}
		\alpha(((1) \ot (1,2) \ot \emptyset \ot (0,1,2) \ot (2)) \ot \tau^{\ot 5}) &=
		\\
		=(-1)^{3\cdot 15}((1) \ot \tau) \ot ((1,2) \ot \tau)& \ot (\emptyset \ot \tau) \ot ((0,1,2) \ot \tau) \ot ((2) \ot \tau)
		\end{split}
	\]
	\end{example}

\begin{definition} Define the chain map
	\[\alpha \colon \cochains(\asimplex^\bullet)^{\ot r} \ot_{\asimplex} \rA(X)_\bullet^{\ot r} \lra \chains(X)^{\ot r}\]
as the composition of the chain isomorphism of Lemma \ref{lemma:3} and the $r$-fold tensor product of the composition of the chain isomorphisms of \eqref{eq:1} and Lemma \ref{lemma:2}.
\end{definition}

From Examples \ref{example:first3}, \ref{example:first3'}, \ref{example:first5} and \ref{example:reord}, we can compute the following:
\begin{example}\label{example:alpha} If $r=3$ and $X=\asimplex^7$ and $\tau = [0,1,2,3,4,5,6,7]$ is the top simplex,
\[
		\alpha((\emptyset \ot (0,6) \ot (0,1,3,4)) \ot \tau^{\ot 3}) = -\tau \ot [1,2,3,4,5] \ot [2,5,6,7]
	\]
	where the sign coming from the permutation of the factors of the tensor product is $(-1)^{10\cdot 7}$ and the sign coming from the $\d_i$'s is negative. If $r=3$ and $X=\asimplex^{19}$ and $\tau = [0,2,3,4,5,6,9]$,
\[
		\alpha((\emptyset \ot (0,6) \ot (0,1,3,4)) \ot \tau^{\ot 3}) = -\tau \ot [2,3,4,5,6] \ot [3,6,9]
	\]
	where the sign coming from the permutation of the factors of the tensor product is $(-1)^{10\cdot 6}$ and the sign coming from the $\d_i$'s is negative. If $r=5$ and $X=\asimplex^2$ and $\tau = [0,1,2]$ is the top simplex,
	\[
		\alpha(((1) \ot (1,2) \ot \emptyset \ot (0,1,2) \ot (2)) \ot \tau^{\ot 5}) = -[0,2] \ot [0] \ot [0,1,2] \ot \emptyset \ot [0,1]
	\]
		where the sign coming from the permutation of the factors of the tensor product is $(-1)^{3\cdot 15}$ and the sign coming from the $\d_i$'s is positive.
	\end{example}
\subsection{The chain complex \texorpdfstring{$\Om(r,n)$}{Omega(r,n)}} If $X$ is an augmented simplicial object, we denote by $\uchains(X)$ the right suspension $\sus{}\ucadenas(X)$. Let $\Om(r,n)$ be the quotient of $\chains(\asimplex^{n}\times \EC_r)$ by the equivalence relation generated as follows: If $U = (u_0,\dots,u_{m-1})\in \uchains(\asimplex^n)$ is an unnormalized generator with $u_i\leq u_{i+1}$ and $A = (a_0,\dots,a_{m-1})\in \uchains(\EC_r)$ is another unnormalized generator, then
\begin{itemize}
	\item $(U,A)\sim (-1)^{\sign{\perm}}(U,A')$ if there is an interval $u_{i-1}<u_i =\dots =u_{i+k}<u_{i+k+1}$ such that the sequences $(a_i,\dots,a_{i+k})$ and $(a'_i,\dots,a'_{i+k})$ differ by a permutation $\perm$ and $A$ and $A'$ agree outside that interval.
\end{itemize}

Given a pair $(U,A)$ as above, define $U_A^i = \{u_j\in U\mid a_j=i\}$. There is an isomorphism
\[
	\beta \colon \Om(r,n) \lra \chains(\asimplex^n)^{\ot r}
\]
that sends a pair $(U,A)$ to $(-1)^{\sign{\perm(A)}}\ U_A^0 \ot \dots \ot U_A^{r-1}$, where $\perm(A)$ is the permutation that arranges $A$ in ascending order without permuting entries with the same value.
\begin{example}\label{example:beta} If $r=3$ and $n \geq 6$:
\[
		\beta((0,0,1,3,4,6),(0,1,0,0,0,1)) = -(0,1,3,4) \ot (0,6) \ot \emptyset
	\]
If $r=5$ and $n \geq 2$:
	\[
		\beta((0,1,1,1,2,2,2),(1,3,4,1,1,3,0)) = -(2) \ot (0,1,2) \ot \emptyset \ot (1,2) \ot (1)
	\]
	\end{example}

Given a pair $(U,A)$ as above, define $A' = (a_0+u_0,\dots,a_m+u_m)$ where the sum is taken modulo $r$. There is also an automorphism
\[
	\gamma \colon \Om(r,n) \lra \Om(r,n)
\]
that sends a pair $(U,A)$ to the pair $(U,A')$.
	\begin{example}\label{example:gamma} If $r=3$ and $n \geq 6$:
	\[
		\gamma((0,0,1,3,4,6),(0,1,2,0,2,1)) = ((0,0,1,3,4,6),(0,1,0,0,0,1))
	\]
If $r=5$ and $n \geq 2$:
	\[
		\gamma((0,1,1,1,2,2,2),(1,2,3,0,4,1,3)) = ((0,1,1,1,2,2,2),(1,3,4,1,1,3,0))
		\]
\end{example}

The compositions
\begin{align*}
	\Omd(r,n)&\overset{\beta^\dd}{\lra} \left(\chains(\asimplex^n)^{\ot r}\right)^\dd \lra \cochains(\asimplex^n)^{\ot r}
	\\
	\Omd(r,n)\overset{\gamma^\dd}{\lra}	\Omd(r,n)&\overset{\beta^\dd}{\lra} \left(\chains(\asimplex^n)^{\ot r}\right)^\dd \lra \cochains(\asimplex^n)^{\ot r}
\end{align*}
	yield two augmented cosimplicial chain complexes with $[n]\mapsto \Omd(r,n)$. The first is denoted $\Omhatd(r,\bullet)$ and the second is denoted simply by $\Omd(r,\bullet)$.
\begin{example}\label{example:betadual3} If $r=3$ and $n=7$:
\[
		\beta^\dd((0,0,1,3,4,6),(0,1,0,0,0,1)) = -\emptyset \ot (0,6) \ot (0,1,3,4)
	\]
	\end{example}
	\begin{example}\label{example:betadual5} If $r=5$ and $n=2$:
	\[
		\beta^\dd((0,1,1,1,2,2,2),(1,3,4,1,1,3,0)) = -(1) \ot (1,2) \ot \emptyset \ot (0,1,2) \ot (2)
	\]
	\end{example}

	The face maps of these cosimplicial chain complexes have degree $-1$ and act as follows: Let $k$ be such that $u_{k-1}<i\leq u_k$, and let $U = (u_0,\dots,u_{m-1})$ and $A = (a_0,\dots,a_{m-1})$, then,
	\[
		\d^i \colon \Omhatd(r,n) \to \susp{}\Omhatd(r,n+1)
	\]
		is given by
	\[\d^i(U, A) = (-1)^{r(i+j+m+n+1)}((u'_0,\dots,u'_{m+r-1}),(a'_0,\dots,a'_{m+r-1}))\] with
\begin{align}\label{eq:Omegahat}
	u'_j &=
	\begin{cases}
		u_j &\text{ if $j<k$} \\
		i & \text{ if $k\leq j < k+r$} \\
		u_{j-r} + 1 & \text{ if $j \geq k+r$.}
	\end{cases}
	&
	a'_j &=
	\begin{cases}
		a_j &\text{ if $j<k$} \\
		j-i & \text{if $k\leq j<k+r$} \\
		a_{j-r} & \text{ if $j \geq k+r$.}
	\end{cases}
\end{align}
whereas
\[
	\d^i \colon \Omd(r,n) \to \susp{}\Omd(r,n+1)
	\]
	is given by
\begin{align*}
	\d^i(U, A) &= (-1)^{r(i + j + m + n + 1)}((u'_0,\dots,u'_{m+r-1}),(a'_0,\dots,a'_{m+r-1}))
\end{align*}
with
\begin{align}\label{eq:Theta}
	u'_j &=
	\begin{cases}
		u_j &\text{ if $j<k$} \\
		i & \text{ if $k\leq j < k+r$} \\
		u_{j-r} + 1 & \text{ if $j \geq k+r$.}
	\end{cases}
	&
	a'_j &=
	\begin{cases}
		a_j &\text{ if $j<k$} \\
		r-i+j-k & \text{if $k\leq j<k+r$} \\
		a_{j-r}-1 & \text{ if $j \geq k+r$.}
	\end{cases}
\end{align}
The cyclic group $\Cyc_r$ acts on $\Om(r,n)$ and $\Omhat(r,n)$ as follows: $\rho(U,A) = (U,\rho A)$, with $\rho(a_0,\dots,a_{m-1}) = (\rho a_0,\dots,\rho a_{m-1})$. On $\Omd(r,n)$ and $\Omhatd(r,n)$ we have the dual action (recall that $\rho(a_0^\vee) = (\rho^{-1}a_0)^\vee$).

\begin{remark}
	Despite the complex $\Om(r,n)$ is defined as a quotient of $\chains(\asimplex^n\times \EC_r)$, the augmented cosimplicial structure on $\Omd(r,n)$ is not induced by any augmented cosimplicial structure on $\cochains(\asimplex^n\times \EC_r)$ for $n \geq -1$.
\end{remark}
As a consequence, there are $\Cyc_r$-equivariant natural isomorphisms
\[\Omd(r,\bullet)\overset{\gamma^\dd}{\lra} \Omhatd(r,\bullet)\overset{\beta^\dd}{\lra} \cochains(\asimplex^\bullet)^{\ot r}.\]
Taking functor tensor products, we obtain $\Cyc_r$-equivariant chain isomorphisms
\[\Omd(r,\bullet) \ot_{\asimplex} \rA(X)_\bullet^{\ot r}\overset{\gamma}{\lra} \Omhatd(r,\bullet) \ot_{\asimplex} \rA(X)_\bullet^{\ot r}\overset{\beta}{\lra} \cochains(\asimplex^\bullet)^{\ot r} \ot_{\asimplex} \rA(X)_\bullet^{\ot r}.\]

\subsection{A connected diagonal}\label{s:mainresult} A \emph{full piece} in a pair $(U,A)\in \Om(r,n)$ is a nondegenerate subpair $(u_i,\dots,u_{i+r-1}),(a_i,\dots,a_{i+r-1})$ of $(U,A)$ such that $u_i = u_{i+1} = \dots = u_{i+r-1}$ (and therefore $a_i,\dots,a_{i+r-1}$ are all distinct).

Let $\Om(r,n)^{\nf} \subset \Om(r,n)$ be the subcomplex generated by those pairs $(U,A)$ without full pieces. If $(U,A)$ has a single full piece, then $\partial(U,A) = \partial^{\nf}(U,A) + \partial^{\f}(U,A)$, where the summands in $\partial^{\nf}(U,A)$ contain no full pieces and the summands in $\partial^{\f}(U,A)$ contain one full piece. Recall that $\Delta$ is the natural diagonal in $\cC$ of \ref{s:categories}, and induces a map $L(\Delta) \colon \rA(X)_n \to \rA(X)_n^{\ot r}$ of degree $(r-1)n$ that we will denote also by $\Delta$. Recall the definition of $\rcadenas(X)$ from \cref{s:connected}, and observe that $\Delta$ also induces a homomorphism $\Delta \colon \rcadenas[rn](X) \to \rA(X)_n^{\ot r}$ of degree $0$.

\renewcommand{\Psiom}{\Psi}

\begin{proposition}\label{prop:omegarm}
	Let $r$ be odd or $r=2$ with $R=\mathbb{F}_2$. Suppose given a family of $\Cyc_r$-equivariant homomorphisms
	\begin{align*}
		\Psiom_*^n \colon \Om(r,n)^{\nf}& \lra \rW(r) & n \geq 0
	\end{align*}
	related by the following equation: if $(U,A)\in \Om[q](r,n)$ has a single full piece,
	\begin{equation}
		 \label{it:1a}
		\Psiom^n_{q-1}(\partial^{\nf} (U,A)) = (-1)^q\theta_{1-r}\circ \Psiom^{n-1}_{q-r}\circ \left(\sum_i (-1)^i \d_i(U,A)\right),
	\end{equation}
	where $\d_i$ is the dual of $\d^i$. Then there is a chain map
	\[
	\Psi \colon \rWd(r) \hotimes \rchains(X) \lra \Omd(r,\bullet) \ot_{\asimplex} \rA(X)_\bullet^{\ot r}
	\]
	defined as $\Psi(e_{-q} \ot \vec{\tau}) = \left(\Psiom^n_q\right)^\vee(e_{-q}) \ot \Delta(\tau)$ if $\tau\in \rA(X)_{n}$.
\end{proposition}
\def\diaglin{\Delta}

\begin{proof}
The suspension homomorphism $\theta_{2k}$ always rises the degree by $2k$, so the dual of $\theta_{2k}$ is $\theta_{-2k}$. We will write $\delta = \partial^\vee$, where $\partial$ is the differential of $\rW(r)$ or $\Om(r,n)$. We have
\begin{align*}
	\delta \Psi(e_{-q}^\dd \ot \vec{\tau})
		&= \delta \left(\Psiom^n\right)^{\vee}(e_{-q}^\dd) \ot \diaglin(\tau)\\
		&= \left(\delta\left(\Psiom^n\right)^{\vee}(e_{-q}^\dd)\right)^{\nf} \ot \diaglin(\tau) + \left(\delta\left(\Psiom^n\right)^{\vee}(e_{-q}^\dd)\right)^{\f} \ot \diaglin(\tau)
\end{align*}
while
\begin{align*}
	\Psi(\delta(e_{-q}^\dd \ot \vec{\tau}))
		&=\Psi(\delta(e_{-q}^\dd) \ot \vec{\tau}) + (-1)^{q} \sum_i (-1)^{i} \Psi(\theta_{r-1}(e_{-q}^\dd) \ot \d_i \vec{\tau})
	\end{align*}
The first summand of the second expression is then $\left(\Psiom^n\right)^\vee(\delta e_{-q}) \ot \diaglin(\tau)$, which equals the first summand of the first expression because $\Psiom^{n}$ is a chain map. For the second summands, we first use the dual of the condition of the Proposition: If $(U,A)$ has a full piece and degree $q+1$,
\begin{align*}
	\left(\delta\left(\Psiom^n\right)^\vee (e_{-q}^\dd)\right)^{\f}(U,A)
	&= \left\langle (-1)^{q+1}\Psiom^n\circ \partial^{\nf}(U,A),e_{q}\right\rangle
	\\
	&= \left\langle (-1)^{q+1}(-1)^{q+1}\theta_{1-r}\circ \Psiom^{n-1}\circ \sum_i (-1)^i \d_i(U,A),e_q\right\rangle
	\\
	&= \left(\sum_i (-1)^i\d^i\circ \left(\Psiom^{n-1}\right)^{\vee}\circ \theta_{r-1}(e_{-q}^\dd)\right)(U,A).
\end{align*}
Hence,
\begin{align*}
	\left(\delta\left(\Psiom^n\right)^{\vee}(e_{-q}^\dd)\right)^{\f} \ot \diaglin(\tau)
		&= \sum_i (-1)^i \d^i\circ \left(\Psiom^{n-1}\right)^\vee\circ \theta_{r-1} (e_{-q}^\dd) \ot \diaglin(\tau)
		\\
		&=(-1)^{r-1-q}\sum_i (-1)^i\left(\Psiom^{n-1}\right)^\vee \circ \theta_{r-1} (e_{-q}^\dd) \ot \d_i\diaglin(\tau)
		\\
		&= (-1)^{r-1-q}\sum_i (-1)^i\left(\Psiom^{n-1}\right)^\vee \circ \theta_{r-1} (e_{-q}^\dd) \ot \diaglin(\d_i\tau)
		\\
	&= (-1)^q\sum_i (-1)^{i} \Psi(\theta_{r-1}(e_{-q}^\dd) \ot \d_i \vec{\tau})).\qedhere
\end{align*}
\end{proof}

In Lemma \ref{lemma:omegar} we will give a criterion to construct a family $\{\Psiom^n\}_n$ satisfying the conditions of Proposition \ref{prop:omegarm}. In Propositions \ref{prop:chain} and \ref{prop:condition} we will reduce this criterion to the construction of a map $f$ satisfying the properties of Construction \ref{cons:1}. Finally, in Proposition \ref{prop:assemblage} and Lemma \ref{lemma:straightening} we will explain how to build the map $f$ from an $r$-cyclic straightening with duality. Altogether they provide the main theorem of this article:

\begin{theorem}\label{thm2:mainthm} Every $r$-cyclic straightening yields a natural $r$-cyclic connected diagonal for augmented semi-simplicial objects in $\cC$
\[
	\mu \colon \rWd(r) \hotimes \rchains(X) \lra \ucadenas(X)^{\ot r}
\]
defined as the composition
	$
	\mu = \alpha\circ\beta\circ\gamma\circ \Psi.
	$
\end{theorem}

\begin{proof}
	To check naturality, suppose $f \colon X \to Y$ is a map of augmented semi-simplicial objects. If $\tau\in \chains[n](X) = \abel(X)_n$, then, we first have that
	\[
		L(\Delta)\circ L(f)(\tau) = L(\Delta\circ f)(\tau) = L(f^{\ot r}\circ \Delta)(\tau) = L(f)^{\ot r}\circ L(\Delta)(\tau).
	\]
	therefore, writing $\Delta$ for $L(\Delta)$ as in the previous proposition, we have
	\begin{align*}
	\mu(e^\dd_{-q} \ot f_*(\vec{\tau}))
		&=	\alpha\circ\beta\circ\gamma\circ \Psi(e_{-q}^\vee \ot f_*(\vec{\tau}))
		\\
		&= \alpha\circ\beta\circ\gamma\circ(\Psiom^n)^\vee(e_{-q}^\vee) \ot \Delta\circ L(f)(\tau))
		\\
		&= \alpha\circ\beta\circ\gamma\circ(\Psiom^n)^\vee(e_{-q}^\vee) \ot L(f)^{\ot r}(\Delta(\tau))
		\\
		&= L(f)^{\ot r}(\alpha\circ\beta\circ\gamma\circ(\Psiom^n)^\vee(e_{-q}^\vee) \ot \Delta(\tau))
		\\
		&= L(f)^{\ot r}(\mu(e^\dd_{-q} \ot \vec{\tau})).
	\end{align*}
	In the penultimate equality we use that $f$ commutes with the face maps.
\end{proof}

\begin{example}\label{ex:omegarn} If $r=3$ and $n=7$, using Examples \ref{example:omegar} and \ref{example:psi3}:
	\[
		\Psiom_6^7((0,0,1,3,4,6),(0,1,2,0,2,1)) = -e_6-\rho^3e_6.
	\]
	Therefore, the coefficient of $((0,0,1,3,4,6),(0,1,2,0,2,1))$ in $\psi(e_6 \ot \tau)$ is $-1$.
 If $r=3$ and $n=6$, using Examples \ref{example:omegar} and \ref{example:psi3}:
	\[
		\Psiom_6^6((0,0,1,3,4,6),(0,1,2,0,2,1)) = -e_6-\rho^3e_6.
	\]
	Therefore, the coefficient of $((0,0,1,3,4,6),(0,1,2,0,2,1))$ in $\psi(e_6 \ot \tau)$ is $-1$.
 If $r=5$ and $n=2$, using Examples \ref{example:omegar} and \ref{example:psi5}:
	\[
		\Psiom_7^2((0,1,1,1,2,2,2),(1,2,3,0,4,1,3)) = 2^2e_7.
	\]
Therefore, the coefficient of $((0,1,1,1,2,2,2),(1,2,3,0,4,1,3))$ in $\psi(e_7 \ot \tau)$ is $2^2$.
\end{example}

\begin{example}\label{ex:omegarnfinal} If $r=3$ and $n=7$, using Examples \ref{example:alpha}, \ref{example:beta}, \ref{example:gamma} and \ref{ex:omegarn} for the computation of $\alpha,\beta,\gamma$ and $\psi$:
	\begin{itemize}
		\item If $r = 3$ and $\tau = [0,1,2,3,4,5,6,7]$, the coefficient of $\tau \ot [1,2,3,4,5] \ot [2,5,6,7]$ in $\mu(e_{-6} \ot \tau)$ is $-1$.
		\item If $r = 3$ and $\tau = [0,2,3,4,5,6,9]$, the coefficient of $\tau \ot [2,3,4,5,6] \ot [3,6,9]$ in $\mu(e_{-6} \ot \tau)$ is $-1$.
		\item If $r = 5$ and $\tau = [0,1,2]$, the coefficient of $[0,2] \ot [0] \ot [0,1,2] \ot \emptyset \ot [0,1]$ in $\mu(e_{-7} \ot \tau)$ is $2^2$.
	\end{itemize}
\end{example}

\subsection{Another bar resolution of \texorpdfstring{$\Cyc_r$}{the cyclic group}}

After identifying the vertices of the standard augmented simplex $\asimplex^{r-1}$ with $\{0,1,\dots,r-1\}$, the agumented simplex carries an action of $\Cyc_r$ that is free away from the empty simplex and the top simplex. Define $\Om(r)$ as the quotient of $\chains(\asimplex^{r-1}) \ot \chains(\asimplex^{r-1}) \ot \dots$ by the relation $(A_0 \ot \dots \ot A_{j-1} \ot \emptyset \ot A_j \ot \dots) \sim (A_0 \ot \dots \ot A_{j-1} \ot A_j \ot \dots)$. Specifically, it is generated on degree $m>0$ by tuples $A = A_0\barra A_1\barra\dots \barra A_k$ with each $A_j$ a non-empty generator of $\chains(\asimplex^{r-1})$ and $\sum_j |A_j| = m$. In degree $0$ it is generated by the empty sequence, which may be denoted $\emptyset$ or $1$, the unit of the ring $R$. Each $A_j$ will be called a \emph{piece} of $A$ and if $A_j$ is of length $r$, it will be called \emph{full piece}. A tuple $A$ as above will be called a \emph{pieced word}. Its differential is $\sum_{j} (-1)^j \diff_j A$, where $\diff_jA$ is the result of removing the $j$th element of $A_0\cup \dots\cup A_k$ from $A$.

Let $\Om(r)^{\nf} \subset \Om(r)$ be the chain subcomplex generated by pieced words without full pieces. Let $\Om(r)^{\f} \subset \Om(r)$ be the chain subcomplex generated by pieced words with at most one full piece. The differential on $\Om(r)^{\f}$ decomposes as $\partial^{\nf}+\partial^{\f}$, where $\partial^{\nf}$ consists on summands without full pieces and $\partial^{\f}$ on summands with a single full piece.

Let $\DDD \colon \Om(r)^{\f} \to \Om(r)^{\nf}$ be the chain map that sends a pieced word without full pieces to zero and a pieced word $A_0\barra\dots\barra A_k$ of degree $m$ with a full piece piece $A_j$ to
\[
\DDD(A_0\barra \dots\barra A_k) = (-1)^{r\left(m+\sum_{k<j} |A_k|\right)} A_0\barra \dots\barra \hat{A}_{j}\barra \rho A_{j+1}\barra \dots\barra \rho A_k
\]
If $(U,A)$ is a generator of $\Om(r,n)$, then the word $A$ has a canonical piece decomposition: a subsequence $(a_i,\dots,a_{i+k})$ of $A$ is a piece if and only if $u_{i-1}<u_i =\dots= u_{i+k}<u_{i+k+1}$. For each $n \geq 0$, there is a chain map
\begin{equation}\label{eq:eta}
    \eta^n \colon \Om(r,n) \to \Om(r)
\end{equation}
that sends a pair $(U,A)$ of degree $m$ to the word $(-1)^{m(n+1)}A$ with its canonical piece decomposition. Additionally, it satisfies the following equations:
\begin{align}\label{eq:lambda2}
\eta^n\circ \partial^{\nf} &= \partial^{\nf}\circ \eta^n
\\
\label{eq:lambda0}
	\eta^{n-1}\left(\sum_i (-1)^{ri}\d_i(U,A)\right) &= \DDD\circ \eta^n(U,A).
\end{align}

\begin{example}\label{example:omegar} To lighten the notation, we represent a piece $A_j = [a_0,\dots,a_{m-1}]$ as $a_0\dots a_{m-1}$: If $r=3$ and $n=7$,
	\[
		\eta^7((0,0,1,3,4,6),(0,1,2,0,2,1)) = (-1)^{6\cdot 8}(01\barra 2\barra 0\barra 2\barra 1).
	\]
	If $r=3$ and $n=6$,
	\[
		\eta^6((0,0,1,3,4,6),(0,1,2,0,2,1)) = (-1)^{6\cdot 7}(01\barra 2\barra 0\barra 2\barra 1).
	\]
	If $r=5$ and $n=2$:
	\[
		\eta^2((0,1,1,1,2,2,2),(1,2,3,0,4,1,3)) = (-1)^{7\cdot 3} 2^{}(1\barra 230\barra 413).
	\]
\end{example}

\begin{lemma}\label{lemma:omegar}
	Let $r$ be odd, and suppose one is given a $\Cyc_r$-equivariant chain map
	\begin{align*}
		\Psiom \colon \Om(r)^{\nf}& \lra \rW(r)
	\end{align*}
	such that if $A$ is a generator of $\Om(r)^{\f}$ with a full piece,
	\begin{equation}\label{eq:lambda}
		(\tilde{r}!)\Psiom_{q-1}(\partial^{\nf} (A)) = (-1)^q\theta_{1-r}\circ\Psiom_{q-r}\circ\DDD(A).
	\end{equation}
	Then the family of maps $\Psiom^n = (\tilde{r}!)^{n+1}\Psiom\circ \eta^n$ satisfies the condition of Proposition \ref{prop:omegarm}.
\end{lemma}

\begin{proof} If $(U,A)\in \Om[q](r,n)$ has a full piece,
\begin{align*}
	\theta_{1-r}\circ \Psiom_{q-r}^{n-1}&\left(\sum_i (-1)^{ri} \d_i(U,A)\right) =
	\\
	&= (\tilde{r}!)^{n}\theta_{1-r}\circ \Psiom_{q-r}\circ \eta^{n-1}\left(\sum_i (-1)^{ri} \d_i(U,A)\right)
	\\
	&\overset{\eqref{eq:lambda0}}{=} (\tilde{r}!)^{n}\theta_{1-r}\circ \Psiom_{q-r}\circ \DDD\circ \eta^{n}(U,A)
	\\
	&\overset{\eqref{eq:lambda}}{=} (-1)^{q}(\tilde{r}!)^{n} \cdot \tilde{r}!\Psiom_{q-1}\circ \partial^{\nf}\circ \eta^{n}(U,A)
	\\
		&\overset{\eqref{eq:lambda2}}{=} (-1)^{q-1}(\tilde{r}!)^{n+1}\Psiom_{q-1}\circ \eta^n\circ \partial^{\nf}(U,A)
	\\
	&= \Psiom^{n}_{q-1}\circ \partial^{\nf}(U,A).\qedhere
\end{align*}
\end{proof}

\begin{remark}
	A particular homomorphism satisfying the hypotheses of Lemma \ref{lemma:omegar} will be found in the next section. Though it is defined with coefficients in $\bZ[\frac{1}{\tilde{r}!}]$, the maps $\{\Psiom^m\}$ obtained in the lemma are well-defined with integer coefficients.
\end{remark}
	% !TEX root = ../oddp.tex

\section{Maps between resolutions of the cyclic group}\label{s:resolutions}

Except for the last part of this section, we assume that $r$ is odd. In the first part of this section we construct a homomorphism from $\chains(\asimplex^{r-1})$ to $\rW(r)$. In the second and third parts we extend this homomorphism to a homomorphism $\Omega_*(r)^{\nf} \to \rW(r)$ satisfying the conditions of Lemma \ref{lemma:omegar}, under the assumption that $r$ is prime. In \cref{ss:milnor} we explain how the above homomorphism yields a chain map from $\chains(\EC_r)$ to $\rW(r)$ when $r$ is odd, and in \cref{ss:even} we discuss the case $r=2$. Recall that $\tilde{r} = \lfloor \frac{r-1}{2}\rfloor$.

\subsection{From the simplex to the minimal resolution}

Through the isomorphism $\Cyc_r \cong \{0,1,\dots,r-1\}$ we endow the latter with a group structure. Given two numbers $a,b\in \{0,1,\dots,r-1\}$, the cyclic difference $b-a$ is computed using this group structure. Given a sequence $(a_0,\dots,a_{q-1}) \subset \{0,1,\dots,r-1\}$, write $a_0\prec a_1\prec \dots\prec a_{q-1}$ if there are integers $\bar{a}_1,\dots,\bar{a}_{q-1}$ that differ from $a_1,\dots,a_{q-1}$ by a multiple of $r$ such that $a_0<\bar{a}_1<\dots<\bar{a}_{q-1}<a_0+r$. In that case, we say that the sequence $(a_0,\dots,a_{q-1})$ is \emph{cyclically ordered}. Define
\begin{align*}
	\phi_{e}(a,b) &= \begin{cases}
		\displaystyle \sum_{i=1}^{\frac{b-a}{2}} \rho^{a+2i} & \text{if $b-a$ is even} \\
		0 & \text{otherwise,}
	\end{cases}
	\\
	\phi_{o}(a,b) &= \begin{cases}
		\displaystyle \sum_{i=1}^{\frac{b-a-1}{2}} \rho^{a+2i+1} & \text{if $b-a$ is odd} \\
		0 & \text{otherwise,}
	\end{cases}
\end{align*}
Let $A= a_0\prec a_1\prec\dots\prec a_{q-1}$ be a cyclically ordered sequence in $\{0,1,\dots,r-1\}$. Say that $a_i$ is \emph{even (resp. \emph{odd}) in the sequence} if $a_{i+1}-a_i$ is an even number (resp. an odd number).
Define the following elements of $\rW(r)$:
\begin{align*}
	\Phi_e(a_0,\dots,a_{q-1}) &= \begin{cases}
		(-1)^{j+1}\phi_e(a_j,a_{j+1})\cdot e_q & \text{if $A$ has a single even entry $a_j$} \\
		0 & \text{otherwise}
	\end{cases}
	\\
	\Phi_o(a_0,\dots,a_{q-1}) &= \begin{cases}
		\sum_{j=0}^{q-1} \phi_o(a_j,a_{j+1})\cdot e_q & \text{if $A$ has no even entries} \\
		0 & \text{otherwise.}
	\end{cases}
\end{align*}

A \emph{cyclic representative} of a generator $[a_0,\dots,a_{q-1}]\in \chains{q}(\partial \asimplex^{r-1})$ is a cyclically ordered sequence $(a_{\perm(0)},\dots,a_{\perm(q-1)})$. Removing an element from the sequence keeps the property of being cyclically ordered. The assignments $\Phi_e(A)$ and $\Phi_o(A)$ induce well defined assigments $\Phi_e[a_0,\dots,a_{q-1}]$ and $\Phi_o[a_0,\dots,a_{q-1}]$ on the generators of $\chains(\asimplex^{r-1})$. In what follows we work always with these cyclic representatives. Let $\varphi(q) = \lfloor\frac{r-q-1}{2}\rfloor$.
\begin{theorem} The linear homomorphism with $\bZ[\frac{1}{\tilde{r}!}]$-coefficients
	\begin{align*}
		\Phi \colon \chains(\partial \asimplex^{r-1})& \lra \rW(r)
		&
		\Phi(A) = \begin{cases}
			e_0 & \text{if $q=0$ and $A = \emptyset$}\\
			\frac{\varphi(q)!}{\tilde{r}!}\Phi_e(A) & \text{if $q$ is even and positive} \\
			\frac{\varphi(q)!}{\tilde{r}!}\Phi_o(A) & \text{if $q$ is odd}
		\end{cases}
	\end{align*}
	is a chain map.
\end{theorem}

\begin{remark}
	If $r=3$, then $\bZ[\frac{1}{\tilde{r}!}] = \bZ$ and if $r$ is prime, then $\mathbb{F}_r$ is a $\bZ[\frac{1}{\tilde{r}!}]$-algebra, hence these coefficients are also allowed.
\end{remark}

\begin{proof} In order to simplify the notation, we will omit the generator $e_q$ in the formulas below.
	Suppose first that $A = [a_0,\dots,a_{q-1}]$ has even length, in which case $\varphi(q) = \varphi(q-1)$. Since $A$ has even length, it must have at least one even entry. If $A$ has more than one even entry, then all sequences in its boundary have at least one even entry, and therefore $\Phi_o(\partial A) = 0 = \partial\Phi_e(A)$. If $A$ has a single even entry $a_j$, then
	\begin{align*}
		\partial \Phi_e(A) &= (-1)^{j+1}\partial \phi_e(a_j,a_{j+1}) = (-1)^{j+1}\left(\rho\phi_e(a_j,a_{j+1}
		) - \phi_e(a_j,a_{j+1})\right).
	\end{align*}
	On the other hand, since $[a_0,\dots,\hat{a}_i,\dots,a_{q-1}]$ will have two even entries unless $i=j$ or $i=j+1$, we have that
	\begin{align*}
		\Phi_o(\partial A)
		&= (-1)^{j}\left(\Phi_o [a_0,\dots,\hat{a}_j,\dots,a_{q-1}]-\Phi_o [a_0,\dots,\hat{a}_{j+1},\dots,a_{q-1}]\right)
	\end{align*}
	Now, using that $\phi_o(a_{j-1},a_{j+1}) = \phi_o(a_{j-1},a_j) + \phi_e(a_j,a_{j+1})$, the first summand equals
	\[
	\phi_o(a_{j-1},a_j) + \phi_e(a_j,a_{j+1})+\sum_{i\neq j-1,j} \phi_o(a_i,a_{i+1}).
	\]
	Using that $\phi_o(a_{j},a_{j+2}) = \phi_e(a_j+1,a_{j+1}+1)+\phi_o(a_{j+1},a_{j+2})$, the second summand equals
	\[
	\phi_e(a_j+1,a_{j+1}+1) + \phi_o(a_{j+1},a_{j+2})+\sum_{i\neq j,j+1} \phi_o(a_i,a_{i+1}).
	\]
	Therefore, most terms cancel and we are left with the value of $\partial \Phi_e[a_0,\dots,a_{q-1}]$.

	Suppose now that $A = [a_0,\dots,a_{q-1}]$ has odd length, in which case $\tilde{r}-\tilde{q}-1 = \tilde{r}-\widetilde{q-1} $. If $A$ has more than two even entries, then $[a_0,\dots,\hat{a}_i,\dots,a_{q-1}]$ has more than one even entry, hence $\Phi_e([a_0,\dots,\hat{a}_i,\dots,a_{q-1}]) = 0$ for all $i$, therefore $\Phi_e(\partial A) = 0 = \partial \Phi_o(A)$. If $A$ has exactly two even entries $a_j,a_k$, then we have two cases: if these entries are not consecutive and $j<k$, then the nontrivial summands in $ \Phi_e(\partial A)$ are
	\begin{align*}
		(-1)^j\Phi_e[a_0,\dots,\hat{a}_j,\dots,a_{q-1}] &= (-1)^j(-1)^{k}\phi_e(a_k,a_{k+1}) \\
		(-1)^{j+1}\Phi_e[a_0,\dots,\hat{a}_{j+1},\dots,a_{q-1}] &= (-1)^{j+1}(-1)^{k}\phi_e(a_k,a_{k+1})
		\\
		(-1)^{k}\Phi_e[a_0,\dots,\hat{a}_k,\dots,a_{q-1}] &= (-1)^k(-1)^{j+1}\phi_e(a_j,a_{j+1}) \\
		(-1)^{k+1}\Phi_e[a_0,\dots,\hat{a}_{k+1},\dots,a_{q-1}] &= (-1)^{k+1}(-1)^{j+1}\phi_e(a_j,a_{j+1})
	\end{align*}
	which cancel in pairs. If the entries are consecutive, then $k=j+1$ and the nontrivial summands in $\Phi_e(\partial A)$ are
	\begin{align*}
		(-1)^j\Phi_e[a_0,\dots,\hat{a}_j,\dots,a_{q-1}] &= (-1)^{j}(-1)^{j+1}\phi_e(a_{j+1},a_{j+2}) \\
		(-1)^{j+1}\Phi_e[a_0,\dots,\hat{a}_{j+1},\dots,a_{q-1}] &= (-1)^{j+1}(-1)^{j+1}\phi_e(a_j,a_{j+2}) \\
		(-1)^{j+2}\Phi_e[a_0,\dots,\hat{a}_{j+2},\dots,a_{q-1}] &= (-1)^{j+2}(-1)^{j+1}\phi_e(a_j,a_{j+1})
	\end{align*}
	and the middle summand cancels with the other two, so $\Phi_e(\partial (A)) = 0 = \partial(\Phi_o(A))$.

	If $A$ has no even entry, then
	\[
	\partial \Phi_o[a_0,\dots,a_{q-1}] = N\left(\sum_{i=0}^{q-1} \phi_o(a_i,a_{i+1})\right) = \frac{r-q}{2}N
	\]
	On the other hand,
	\[
	\Phi_e(\partial A)
	= \sum_{i=0}^{q-1} \phi_e(a_i,a_{i+2})
	= N
	\]
	This completes the proof, since $\frac{r-q}{2} = \frac{\varphi(q-1)!}{\varphi(q)!}$.
\end{proof}

\begin{example}\label{example:Phi}
If $r=3$, we have:
\begin{align*}
\Phi([0]) &= e_1
&
\Phi([1,2]) &= \rho e_2
\end{align*}
If $r=5$, we have:
\begin{align*}
\Phi([0,1,2]) &= 2^{-1}e_3
&
\Phi([0,1]) &= 2^{-1}(e_2 + \rho^3e_2)
&
\Phi([0,3]) &= 2^{-1}e_2.
\end{align*}
\end{example}

\begin{remark}\label{remark:phidual}
	For each $0\leq q<r-1$, let $L_q \subset (\partial \asimplex^{r-1})_q$ be the set of those simplices $[a_0,\dots, a_{q-1}]$ with increasing entries (in the total order) such that $a_{2i}$ is even and different from $r-1$, and $a_{2i+1}$ is odd. Then the linear dual of $\Phi$ has the following expression
	\begin{equation}\label{eq:111}
		\Phi^\vee(e_{-q}) = \frac{\varphi(q)!}{\tilde{r}!}\sum_{A\in L_q} A^\vee.
	\end{equation}
\end{remark}
	% !TEX root = ../oddp.tex

\subsection{From \texorpdfstring{$\Omega_*(r)$}{Omega(r)} to the minimal resolution}\label{ss:mapf}

The \emph{pivotal piece} of a pieced word $A_0\barra \dots\barra A_k$ is the leftmost piece $A_j$ such that $A_{j+1}*\dots*A_k$ has degree smaller than $r$ (i.e., $A_{j+1}\cup\dots\cup A_k$ has less than $r$ entries).

\begin{example}
	If $r=5$, the pivotal piece of the pieced word $A=01|24|013|12|4$ is $A_1 = 013$, while $A_2 = 12|4$ and $\hat{A} = 01|24$. If $r=3$, the pivotal piece of the pieced word $01|20|1|02|1|20|12$ is $20$, while $A_2 = 12$ and $\hat{A} = 01|20|1|02|1$.
\end{example}

The join homomorphism $\chains(\asimplex^{r-1}\join \asimplex^{r-1}) \to \chains(\asimplex^{r-1})$ of Definition \ref{d:join_product} induces an endomorphism
\[
\kappa \colon \Omega_*(r)^{\nf} \lra \Omega_*(r)^{\nf}
\]
that sends a pieced word $A_0\barra \dots\barra A_k$ with pivotal piece $A_j$ to the pieced word $(-1)^{\lambda(A_{j+1},\dots,A_k)}A_0\barra \dots\barra A_j\barra A_{j+1}\join\dots\join A_k$, where the sign is that of the permutation that orders the entries of $A_{j+1},\dots,A_k$.

The cyclic group $\Cyc_r$ acts on $\partial \asimplex^{r-1}$ by permuting its vertices forwards. Let $\sigma = [0,1,\dots,r-1]$ be the top simplex of $\asimplex^{r-1}$ and let
\begin{align*}
	\iota_1 \colon \sus{r-1}\chains(\partial \simplex^{r-1})& \lra \chains(\partial\simplex^{r-1}) \ot \chains(\partial\simplex^{r-1})\\
	\iota_2 \colon \sus{r-1}\chains(\partial \simplex^{r-1})& \lra \chains(\partial\simplex^{r-1}) \ot \chains(\partial\simplex^{r-1})
\end{align*}
be the $\Cyc_r$-equivariant chain maps of degree $r-1$ given by
\begin{align*}
	\iota_1(\tau) &= \tau \ot \partial \sigma &
	\iota_2(\tau) &= \partial \sigma \ot \tau
\end{align*}

\begin{construction}\label{cons:1} Let $f \colon \chains(\partial\simplex^{r-1}) \ot \chains(\partial\simplex^{r-1}) \to \sus{r-1}\chains(\partial\asimplex^{r-1})$ be a $\Cyc_r$-equivariant chain map such that
\renewcommand{\theenumi}{\roman{enumi}}
\begin{enumerate}
	\item\label{cond:1} $f\circ \iota_1 = \Id$
	\item\label{cond:2} $f\circ \iota_2 = \rho$.
	\item\label{cond:3} if $\tau_1 \ot \tau_2$ has degree $r$, then $ \partial f(\tau_1 \ot \tau_2) =
	\begin{cases}
		(-1)^{\lambda(\tau_1,\tau_2)}\emptyset & \text{if $\tau_1^c=\tau_2$} \\
		0 & \text{otherwise},
	\end{cases}$
	where $\lambda(\tau_1,\tau_2)$ is the sign of the permutation that orders $\tau_1\cup \tau_2$.
\end{enumerate}
\end{construction}

\begin{remark}\label{remark:3prime}
	Condition \eqref{cond:3} is implied by the following condition:
	\begin{itemize}
		\item[(iii')] If $\tau_1 \ot \tau_2$ has degree $r$, then $f(\tau_1 \ot \tau_2)$ is $(-1)^{\lambda(\tau_1,\tau_2)} v$ for some vertex of $\partial \asimplex^{r-1}$ if $\tau_1\cup\tau_2$ is non-degenerate and $0$ otherwise.
	\end{itemize}
\end{remark}

\begin{proposition} Let $f$ be a chain map as in Construction \ref{cons:1}, satisfying Conditions \eqref{cond:1} and \eqref{cond:3}. Then, there is a chain map
	\[
	S \colon \Omega_*(r)^{\nf} \lra \sus{r-1}\Omega_*(r)^{\nf}
	\]
	that sends a pieced word $A = A_0\barra \dots\barra A_k$ with pivotal piece $A_j$ to the pieced word $ A_0\barra \dots\barra f(A_j \ot \kappa(A_{j+1}\barra \dots\barra A_k))$.
\end{proposition}

\begin{proof}
	Replacing the word $A$ by $\kappa(A)$, we may assume that $A_{j+1} = A_k$. Let $\bar{A} = A_0\barra \dots\barra A_{j-1}$ and let $\ell$ be the degree of $\bar{A}$ (i.e., the number of entries of $A$). Then 
	\begin{equation}\label{eq:931}
		\partial S(A) = \partial \bar{A}\mid f(A_{j} \ot A_{j+1}) + (-1)^{\ell}\bar{A}\barra \partial f(A_{j} \ot A_{j+1}).
	\end{equation}
	We distinguish two cases: if $A_j\cup A_{j+1}$ has length at least $r+1$,
	\[S(\partial A) = \partial \bar{A}\barra f(A_j \ot A_{j+1}) + (-1)^{\ell}\bar{A}\barra f(\partial (A_j \ot A_{j+1})),\]
	which equals the previous sum. If $A_j\cup A_{j+1}$ has length $r$, let $\check{A} = A_0\mid \dots\mid A_{j-2}$. Then
	\begin{equation}\label{eq:933}
		S(\partial A) = \partial \bar{A}\barra f(A_j \ot A_{j+1}) + (-1)^\ell \check{A}\barra f(A_{j-1} \ot \partial(A_j*A_{j+1})).
	\end{equation}
	Let $\pm$ be the sign of the permutation that orders $A_j*A_{j+1}$. By condition \eqref{cond:3}, the last summand of \eqref{eq:931} is zero or $\pm(-1)^\ell\bar{A}$ depending on whether $A_j*A_{j+1}$ is degenerate or not, which by condition \eqref{cond:1} is equivalent to $f(A_{j-1},\partial(A_j*A_{j+1}))$ being $0$ or $\pm(-1)^{\ell}A_{j-1}$, and therefore equivalent to the last summand of \eqref{eq:933} being zero or $\pm(-1)^{\ell}\bar{A}$.
\end{proof}

\begin{definition}\label{def:psiom}
	Let $f$ be a chain map as in Construction \ref{cons:1}, satisfying Conditions \eqref{cond:1} and \eqref{cond:3}. Define a homomorphism $\Psi \colon \Omega_*(r)^{\nf} \ot \to \rW(r) \ot \bZ[\frac{1}{\tilde{r}!}]$ recursively as
	\[\Psiom_q(A) = \begin{cases} \Phi_q(\kappa(A)) & \text{if $q\leq r-1$} \\
		\frac{1}{\tilde{r}!}\theta_{r-1}\circ \Psiom_{q-r+1}\circ S(A) & \text{if $q \geq r$,}\end{cases}\]
\end{definition}

\begin{example}\label{example:psi3}
Let $r=3$. We will use the computations of Examples \ref{example:Phi}, \ref{example:f3_1} and \ref{example:f3_2} with the $f$ constructed in \cref{s:assembly}. To lighten the notation, we will write each block $A_j = [a_0,\dots,a_{m-1}]$ as $a_0a_1\dots a_{m-1}$, so $[0,1]$ is written $01$.
	\begin{align*}
		\Psi_6(01\barra 2\barra 0\barra 2\barra 1)
		&= \theta_2\circ \Psi_4\circ S(01\barra 2\barra 0\barra 2\barra 1)
		\\
		&= \theta_2\circ \Psi_4(01\barra 2\barra f([0] \ot [2,1]))
		\\
		&= -\theta_2\circ \Psi_4\circ S(01\barra 2\barra 0)
		\\
		&= -\theta_4\circ \Psi_2(f([0,1] \ot [2,0]))
		\\
		&= -\theta_4\circ \Psi_2(01)
		\\
		&= -\theta_4\circ \Phi_2(01)
		\\
		&= -\theta_4(e_2 + \rho^3e_2)
		\\
		&= -e_6 - \rho^3e_6
	\end{align*}
\end{example}

\begin{example}\label{example:psi5}
Let $r=5$. We will use the computations of Examples \ref{example:Phi} and \ref{example:f5_1} with the $f$ constructed in \cref{s:assembly}. We use the same lightened notation as in the previous example
	\begin{align*}
		\Psi_7(1\barra 230\barra 413)
		&= 2^{-1}\theta_4\circ \Psi_3\circ S(1\barra 230\barra 413)
		\\
		&= 2^{-1}\theta_4\circ \Psi_3(1\barra f([2,3,0] \ot [4,1,3]))
		\\
		&= 2^{-1}\theta_4\circ \Psi_3(1\barra 02)
		\\
		&= 2^{-1}\theta_4\circ \Phi_3([1,0,2])
		\\
		&= -2^{-1}\theta_4(e_3)
		\\
		&= -2^{-1}e_7
	\end{align*}
\end{example}

\begin{proposition}\label{prop:chain}
	$\Psiom$ is a chain map. 
\end{proposition}

\begin{proof}
	$\Psiom_*$ is already a chain map up to degree $*\leq r-1$. To check that it commutes with the differential in degree $r$, we distinguish two cases:

	First, if $\Psiom(A) = 0$, then $A$ has repeated entries. If it has one repeated entry, then the image by $\Psiom$ of its boundary has two cancelling terms, while if it has more repeated entries, it is zero on the nose.

	Second, if $\Psiom(A)\neq 0$, then $A$ has no repetitions. We may assume, as before, that $A = A_1|A_2$ has only two pieces, and write $A_1*A_2$ for their join and $\diff_iA_1*A_2$ for the result of removing the $i$th entry. Then:
	\begin{align*}
		\Psi_{r-1}(\partial A)
		&= \sum_i (-1)^i \Phi(\diff_i(A_1*A_2))
		\\
		&= (-1)^{\lambda(A_1,A_2)}\sum_i (-1)^i(-1)^{i}{\textstyle\frac{1}{\tilde{r}!}}\rho^{i+1} e_{r-1}
		\\
		&= (-1)^{\lambda(A_1,A_2)} {\textstyle\frac{1}{\tilde{r}!}} N(e_{r-1}).
	\end{align*}
	On the other hand,
	\begin{align*}
		\partial \circ \Psi_r(A) &= {\textstyle\frac{1}{\tilde{r}!}}\partial \circ \theta_{r-1}\circ \Psi_1\circ S(A)
		\\
		&= {\textstyle\frac{1}{\tilde{r}!}}\theta_{r-1}\circ \Psi_0\circ \partial\circ S(A)
		\\
		&= {\textstyle\frac{1}{\tilde{r}!}}\theta_{r-1}\circ \Psi_0\circ \partial\circ f(A_1 \ot A_2)
		\\
		&= (-1)^{\lambda(A_1,A_2)}{\textstyle\frac{1}{\tilde{r}!}} \theta_{r-1}(e_0)
		\\
		&= (-1)^{\lambda(A_1,A_2)}{\textstyle\frac{1}{\tilde{r}!}} N(e_{r-1}).
	\end{align*}
	Finally, to check that $\Psi$ commutes with the differential in degrees $q>r$, one proceeds by induction.
\end{proof}

\begin{proposition}\label{prop:condition}
	Let $f$ be a chain map as in Construction \ref{cons:1}, satisfying Conditions \eqref{cond:1}, \eqref{cond:2} and \eqref{cond:3}. If $\Psi$ is defined as in Definition \ref{def:psiom}, and $A\in \Omega_*(r)$ has a full piece,
	\[
	\tilde{r}!\Psiom_{q-1}(\partial^{\nf} (A)) = (-1)^q\theta_{r-1}\circ \Psiom_{q-r}\circ\DDD(A).
	\]
\end{proposition}

\begin{proof}
	We will prove it by induction on the position of the full piece from the right. Since $A$ has one full piece, $q \geq r$. We assume without loss of generality that the full piece is $[0,1,\dots,r-1]$ (i.e., it is ordered), and we distinguish three cases:

	If the full piece is the last piece, write $A=\bar{A}\barra A_1\barra A_2$ with $A_2 = \sigma$ the full piece and $A_1$ the penultimate piece, and let $\ell$ be the length of $\bar{A}*A_1$. Then by Condition \eqref{cond:1}:
	\begin{align*}
		\tilde{r}!\Psiom_{q-1}(\partial^{\nf} (A)) &=
		\theta_{r-1} \circ \Psiom_{q-r}\circ S (\partial^{\nf}(A)) \\
		&= (-1)^{\ell} \theta_{r-1}\circ\Psiom_{q-r}\circ S(\bar{A}\barra A_1\barra \partial A_2) \\
		&= (-1)^{\ell} \theta_{r-1}\circ\Psiom_{q-r}(\bar{A}\barra f(A_1 \ot \partial \sigma)) \\
		&= (-1)^{\ell} \theta_{r-1}\circ\Psiom_{q-r}(\bar{A}\barra A_1) \\
		&= (-1)^q\theta_{r-1}\circ \Psiom_{q-r}\circ\DDD(A).
	\end{align*}
	If the full piece is the penultimate piece, write $A=\bar{A}\barra A_1\barra A_2$ with $A_1$ the full piece and $\ell$ for the length of $\bar{A}$. Then by Condition \eqref{cond:2}:
	\begin{align*}
		\tilde{r}!\Psiom_{q-1}(\partial^{\nf} A) &=
		\theta_{r-1}\circ \Psiom_{q-r}\circ S(\partial^{\nf}(A)) \\
		&= (-1)^{\ell} \theta_{r-1}\circ \Psiom_{q-r}\circ S(\bar{A}\barra \partial A_1\barra A_2) \\
		&= (-1)^{\ell} \theta_{r-1}\circ \Psiom_{q-r}(\bar{A}\barra f(\partial A_1 \ot A_2)) \\
		&= (-1)^{\ell} \theta_{r-1}\circ \Psiom_{q-r}(\bar{A}\barra \rho(A_2)) \\
		&= (-1)^q\theta_{r-1}\circ \Psiom_{q-r}\circ \DDD(A).
	\end{align*}
	Otherwise, $q \geq 2r$, and we reason by induction as follows:
	\begin{align*}
		(\tilde{r}!)^2\Psiom_{q-1}(\partial^{\nf} A)
		&= (-1)^q\tilde{r}!\theta_{r-1}\circ \Psiom_{q-r}\circ S(\partial^{\nf} A) \\
		&= (-1)^q\tilde{r}!\theta_{r-1}\circ \Psiom_{q-r}(\partial^{\nf} SA) \\
		&= (-1)^{q+q-r}\theta_{r-1}\circ \theta_{r-1}\circ \Psiom_{q-2r+1}\circ \DDD (SA)) \\
		&= (-1)^{q+q-r}\theta_{r-1}\circ \theta_{r-1}\circ \Psiom_{q-2r+1}\circ S\circ \DDD (A) \\
		&= (-1)^{q}\tilde{r}!\theta_{r-1}\circ \Psiom_{q-r}\circ \DDD (A)\qedhere
	\end{align*}
\end{proof}
	% !TEX root = ../oddp.tex

\section{Cyclic straightenings and subdivisions of simplices}\label{section:atlast}

\subsection{Barycentric subdivisions}\label{s:assembly}

Recall that the barycentric subdivision $\sd \partial\simplex^n$ of the boundary of the standard simplex $\simplex^n$ is the simplicial complex that has one vertex for each non-empty face of $\simplex^n$ and one face $(\tau_0,\dots,\tau_k)$ of dimension $k$ for every ascending chain $\tau_0 \subset \tau_1\subset\dots \subset \tau_k$ of simplices of $\partial \simplex^n$. We will denote the face $\tau_0 \subset \dots \subset \tau_k)$ as $(\bar{\tau}_0|\bar{\tau}_1|\dots|\bar{\tau}_k)$, where $\bar{\tau}_i = \tau_i\smallsetminus \tau_{i-1}$ is the face such that the geometric join $\bar{\tau}_i\star \tau_{i-1}$ is the face $\tau_i$. With this notation, the differential on $\cadenas(\sd \partial \simplex^n)$ becomes
\[
\partial(\bar{\tau}_0|\bar{\tau}_1|\dots|\bar{\tau}_k) = \sum_{i=0}^{k-1} (-1)^i(\bar{\tau}_0|\dots|\bar{\tau}_i\star \bar{\tau}_{i+1}|\dots |\bar{\tau}_k) + (-1)^k (\bar{\tau}_0|\dots|\bar{\tau}_{k-1}).
\]

The \emph{pair barycentric subdivision} $\Psd \partial\simplex^n$ of $\partial \simplex^n$ is a cubulation of $\partial \simplex^n$ with the same vertices as $\sd \partial \simplex^n$ and one face for each pair $(b,a)$ of faces of $\partial\simplex^n$ such that $b \subset a$ (cf. \cite{Rounds2010}). Geometrically, that face is the union of all the faces of dimension $|a|-|b|$ of the barycentric subdivision that correspond to ascending chains $b = \tau_0 \subset \dots \subset \tau_k= a$. Interpreting $b$ as a dual cochain, its chain complex $\cadenas(\Psd \partial\simplex^n)$ is isomorphic to the tensor product $\cochains(\partial\simplex^n) \ot \chains(\partial\simplex^n)$ modulo the pairs $b \ot a$ such that the support of $b$ is not contained in $a$.

There are chain maps
\[
\xymatrix{
	&\susp{-n-1}\chains(\partial\simplex^n) \ot \chains(\partial\simplex^n)\ar[d]^h& \\
	\cadenas(\partial\simplex^n)\ar[r]^{s_*^{\simplex}}\ar@/_2.0pc/[rr]^{s_*} & \cadenas(\Psd \partial\simplex^n)\ar[r]^{s^{\Psd}_*} & \cadenas(\sd\partial\simplex^n)
}
\]
defined as follows: if $[a_0,\dots,a_{k-1}]\in \cadenas(\partial\simplex^n)$ is an ordered representative of a generator,
\begin{align*}
	s_*([a_0,\dots,a_{k-1}]) &= \sum_{\perm\in \Sigma_{k}} (-1)^{\sign{\perm}}(a_{\perm(0)}|a_{\perm(1)}|\dots|a_{\perm(k-1)}),
	\\
	s_*^{\simplex}([a_0,\dots,a_{k-1}]) &= (-1)^{k}\sum_{j=0}^{k-1} (a_j) \ot [a_0,\dots,a_{k-1}].
\end{align*}
If $b = (b_0,\dots,b_{m-1})\in \cochains(\partial\simplex^n)$ and $a = [b_0,\dots,b_{m-1},c_0,\dots,c_{\ell-1}]\in \chains(\partial\simplex^n)$ with $k = m+\ell$, then
\begin{equation}\label{eq:sP}
	s_*^{\Psd}(b \ot a) = (-1)^{\ell+1}\sum_{\perm\in \Sigma_{k+1}} (-1)^{\sign{\perm}} (b|c_{\perm(0)}|c_{\perm(1)}|\dots|c_{\perm(\ell-1)})
\end{equation}
Notice that the letter $b$ plays a different role in each side: on the left, it is a dual generator of the normalized cochain complex, thus it makes sense to permute its entries changing the generator by a sign. On the right, $b$ is a vertex of $\sd\partial\simplex^n$.

\begin{remark}
	Let $\twist \colon A \ot B \to B \ot A$ the chain map that sends $a \ot b$ to $(-1)^{|a||b|}b \ot a$. Then
	\begin{align*}
		(f \ot g)\circ \twist &= (-1)^{|f||g|}\twist \circ (g \ot f)
	\end{align*}
\end{remark}

Recall the Alexander duality isomorphism $\alex \colon \chains(\simplex^n) \to \susp{n+1}\cochains(\simplex^{n+1})$, that sends a generator $\tau$ of degree $k$ to $(-1)^{\lambda(\tau,\tau^c)} (\tau^c)^{\vee}$. Define the chain map $h$ as
\[
h( b \ot a) = (\alex \ot \id)(b \ot a).
\]
where $b$ has degree $m$ and $a$ has degree $k$.

Consider the following endomorphisms of the chain complexes of the barycentric and the pair subdivision
\begin{align*}
	\alex \colon \sd \partial\simplex^n& \lra \sd \partial\simplex^n,
	&
	\alex \colon \Psd \partial\simplex^n& \lra \Psd \partial\simplex^n.
\end{align*}
The first one sends an ascending chain of simplices $\tau_0 \subset \dots \subset \tau_{k-1}$ to the ascending chain of their complementary simplices $(-1)^{\binom{k}{2}} \tau_{k-1}^c \subset \dots\subset\tau_0^c$, which in our notation becomes $(\bar{\tau}_0|\bar{\tau}_1|\dots|\bar{\tau}_{k-1})\mapsto (-1)^{\binom{k}{2}}(\tau_{k-1}^c|\bar{\tau}_{k-1}|\dots|\bar{\tau}_1)$. The second one sends $b \ot a$ to $\twist\circ \left(\alex^{-1} \ot \alex\right)(b \ot a)$.
\begin{lemma} The following diagram commutes, with each square commuting up to the indicated sign:
	\begin{equation}\label{diag:newdiag}
		\xymatrix{
			\susp{-n-1}\chains(\partial \simplex^n*\partial\asimplex^n)\ar[r]^-h\ar[d]^{\twist} \ar@{}[dr] | {(-1)^{n+1}}&\cadenas(\Psd \partial \asimplex^n) \ar[r]^{s_*^{\Psd}}\ar[d]_{\alex} \ar@{}[dr] | {(-1)^{n+1}}& \cadenas(\sd \partial \asimplex^n)\ar[d]^{\alex}\\
			\susp{-n-1}\chains(\partial \asimplex^n*\partial\simplex^n)\ar[r]^-h &\cadenas(\Psd \partial \asimplex^n) \ar[r]^{s_*^{\Psd}}& \cadenas(\sd \partial \asimplex^n).
		}
	\end{equation}
\end{lemma}

\begin{proof}
	For the first square, we have
	\begin{align*}
		\twist\circ(\Lambda^{-1} \ot \Lambda)\circ (\Lambda \ot \id)
		&= (-1)^{n+1} \twist\circ(\id \ot \Lambda)
		= (-1)^{n+1} h.
	\end{align*}
	For the second square, observe that, with the notation of \eqref{eq:sP}, if $b = (b_0,\dots,b_{m-1})$ is a generator of $\cocadenas(\PsdSimp)$ and $a = [b_0,\dots,b_{m-1},c_0,\dots,c_{\ell-1}]$ is a generator of $\cadenas(\PsdSimp)$, with $k=m+\ell$,
	\begin{equation}\label{eq:lastsign}
		\begin{split}
			\alex\circ s_*^{\Psd} (b \ot a)
			&= (-1)^{\ell+1}\sum_{\perm} (-1)^{\sign{\perm}} \alex(b|c_{\perm(0)}|\dots|c_{\perm(\ell-1)})
			\\
			&= -\sum_{\perm} (-1)^{\sign{\perm}} (a^c|c_{\perm(0)}|\dots|c_{\perm(\ell-1)})
		\end{split}
	\end{equation}
	(be aware that the last equality equates the summand indexed by a permutation $\perm$ with the summand indexed by the permutation $\perm\circ \mathrm{rev}$, where $\mathrm{rev}$ is the permutation that sends $(0,1,\dots,\ell-2,\ell-1)$ to $(\ell-1,\ell-2,\dots,1,0)$, and the extra sign of the permutation $\mathrm{rev}$ is the sign $(-1)^{\binom{\ell}{2}}$ carried by $\alex$, and an additional sign $(-1)^{\ell}$ is obtained by moving $a^c$ from the last position to the first position over $(c_0,\dots,c_{\ell-1})$).

	Suppose now that both $b=[b_0,\dots,b_{m-1}]$ and $c=[c_0,\dots,c_{\ell-1}]$ are ordered, and that $a = [a_0,\dots,a_k]$ is the result of ordering $[b_0,\dots,b_{m-1},c_0,\dots,c_{\ell-1}]$. Define $x=[x_0,\dots,x_{n-k}]$ as the ordered complement of $a$, and let $y=[y_0,\dots,y_{n-m}]$ be the ordered complement of $b$. Recall that we denote by $\lambda(a,b)$ the sign that orders $a\cup b$. Then $s_*^{\Psd}\circ \Lambda$ produces the following sign:
	\[
	\xymatrix{
		(b_0,\dots,b_{m-1}) \ot [b_0,\dots,b_{m-1},c_0,\dots,c_{\ell-1}]
		\ar[d]^-{(-1)^{\lambda(b,c)}}_-{\sim}
		\\
		(b_0,\dots,b_{m-1}) \ot [a_0,\dots,a_{k-1}]
		\ar[d]^{(-1)^{m(n+1)+\lambda(a,x)}}_{\id \ot \Lambda}
		\\
		(b_0,\dots,b_{m-1}) \ot (x_0,\dots,x_{n-k})
		\ar[d]^-{(-1)^{\lambda(y,b)}}_-{\Lambda^{-1} \ot \id}
		\\
		[y_0,\dots,y_{n-m}] \ot (x_0,\dots,x_{n-k})
		\ar[d]^-{(-1)^{(n-m+1)(n-k+1)}}_-{\twist}
		\\
		(x_0,\dots,x_{n-k}) \ot [y_0,\dots,y_{n-m}]
		\ar[d]^-{(-1)^{\lambda(x,c)}}_-{\sim}
		\\
		(x_0,\dots,x_{n-k}) \ot [x_0,\dots,x_{n-k},c_0,\dots,c_{\ell-1}]
		\ar[d]^-{(-1)^{\ell+1}}_-{s_*^{\Psd}}
		\\
		\sum_{\perm} (-1)^{\sign{\perm}}[x,c_{\perm(0)},\dots,c_{\perm(\ell-1)}]
	}
	\]
	Now, $\lambda(b,c) + \lambda(a,x)$ computes the sign of the permutation that orders
	\[
	[b_0,\dots,b_{m-1},c_0,\dots,c_{\ell-1},x_{0},\dots,x_{n-k}],
	\]
	and so does $\lambda(c,x)+\lambda(b,y)$. Since
	\begin{align*}
		\lambda(c,x) &\equiv \lambda(x,c) + \ell(n+1-k)
		&
		\lambda(b,y) &\equiv \lambda(y,b) + m(n+1-m),
	\end{align*}
	we may replace $\lambda(b,c)+\lambda(a,c) + \lambda(x,c) + \lambda(y,b)$ by $\ell(n+1-k)+m(n+1-m)$. Using that $k=m+\ell$, we get that the total sign is $(-1)^{n}$. Together with the $-1$ sign obtained in \eqref{eq:lastsign} we obtain that the second square also commutes up to $(-1)^{n+1}$.
\end{proof}

\subsection{Cyclic straightenings and the construction of the map \texorpdfstring{$f$}{f}}

\begin{definition}
	An \emph{assemblage map} is a simplicial map $g \colon \sd \partial \simplex^n \to \partial \simplex^n$ such that the composition of chain maps $g_*\circ s_*$ is the identity.
\end{definition}

\begin{definition}
	An \emph{$r$-cyclic assemblage map} is an assemblage map $g \colon \sd \partial \simplex^{r-1} \to \partial \simplex^{r-1}$ that is cyclic with respect to the forward action of the cyclic group $\Cyc_{r}$.
\end{definition}

\begin{definition}
	An \emph{$r$-cyclic assemblage map with duality} is an $r$-cyclic assemblage map $g$ such that the composition
	\[\sd \partial \simplex^{r-1} \overset{\alex}{\lra} \sd \partial\simplex^{r-1} \overset{g}{\lra} \partial\simplex^{r-1}\overset{\rho^{-1}}{\lra} \partial\simplex^{r-1} \]
	is also an assemblage map.
\end{definition}

\begin{proposition}\label{prop:assemblage}
	Let $r$ be odd. If $g$ is an $r$-cyclic assemblage map with duality, then the composition
	\[
	\xymatrix{
		\sus{-r}\left(\chains(\asimplex^{r-1}) \ot \chains(\asimplex^{r-1})\right)\ar[d]^{(-1)^{r(k+m)}}&&
		\chains(\asimplex^{r-1})
		\\
		\susp{-r}\left(\chains(\asimplex^{r-1}) \ot \chains(\asimplex^{r-1})\right)\ar[d]^-{\twist} &&
		\cadenas(\asimplex^{r-1})\ar[u]
		\\
		\susp{-r}\left(\chains(\asimplex^{r-1}) \ot \chains(\asimplex^{r-1})\right)\ar[r]^-{h} &
		\Psd(\asimplex^{r-1}) \ar[r]^{s_*^{\Psd}} &
		\cadenas(\sd\asimplex^{r-1})\ar[u]_{g_*}
		}
	\]
	satisfies the conditions of Construction \ref{cons:1}.
\end{proposition}

\begin{proof}
	Let $\sigma = [0,1,\dots,r-1]$ be the top simplex of $\asimplex^{r-1}$. For condition \eqref{cond:1}, Since $r-1$ is even,
	\begin{align*}
		h(\partial\sigma \ot \tau) &= \Lambda(\partial \sigma) \ot \tau = (-1)^{r}\partial \Lambda(\sigma) \ot \tau
		\\&= (-1)^{r-1}\partial\emptyset \ot \tau = \sum_{v\in \tau} v \ot \tau,
	\end{align*}
	hence, since the degree of $\partial \sigma$ is $r-1$,
	\begin{align*}
		f(\tau \ot \partial\sigma)
		&=(-1)^{r(k+r-1)} g_*\circ s_*^{\Psd}\circ h \circ \mathrm{twist}(\tau \ot \partial\sigma)\\
		&= (-1)^{k} g_*\circ s_*^{\Psd}\circ h (\partial\sigma \ot \tau)\\
		&= g_*\circ s_*^{\Psd}\left((-1)^k\sum_{v\in \tau} v \ot \tau\right) \\
		&= g_*\circ s_*^{\Psd}\circ s_*^{\simplex}(\tau) \\
		&= \tau.
	\end{align*}
	For condition \eqref{cond:2}, since diagram \eqref{diag:newdiag} commutes, using that $\rho^{-1}\circ g_*\circ \Lambda$ is also an assemblage map,
	\begin{align*}
		f(\partial\sigma \ot \tau)
		&= (-1)^{r(r-1+k)}g_*\circ s_*^{\Psd}\circ h \circ \mathrm{twist}(\partial\sigma \ot \tau)
		\\
		&= (-1)^{r(r-1+k)} g_*\circ \Lambda\circ s_*^{\Psd}\circ h (\partial \sigma \ot \tau)
		\\
		&= (-1)^{r(r-1+k)} g_*\circ \Lambda\circ s_*^{\Psd}\circ h \circ \mathrm{twist}(\tau \ot \partial\sigma)
		\\
		&= (-1)^{r(r-1+k)}\rho\circ \rho^{-1} \circ g_*\circ \Lambda\circ s_*^{\Psd}\circ h \circ \mathrm{twist}(\tau \ot \partial\sigma)
		\\
		&= \rho(\tau).
	\end{align*}
	Finally, we prove the stronger Condition (iii') of Remark \ref{remark:3prime}. Let $\tau_1$ and $\tau_2$ have degree $m$ and $k$ with $m+k=r$. If $\tau_1*\tau_2$ is non-degenerate, then $\tau_2 = \tau_1^c$ and
	\begin{align*}
		h(\tau_1 \ot \tau_2)
		&= (-1)^{r\cdot r} g_*\circ s_*^{\Psd}\circ h\circ \twist(\tau_1 \ot \tau_2)
		\\
		&= (-1)^{r + mk} g_*\circ s_*^{\Psd}\circ h(\tau_2 \ot \tau_1)
		\\
		&= (-1)^{r + mk + \lambda(\tau_2,\tau_1)} g_*\circ s_*^{\Psd}(\tau_1 \ot \tau_1)
		\\
		&= (-1)^{r + mk + \lambda(\tau_2,\tau_1) + 1} g_*\circ s_*^{\Psd}(-\tau_1 \ot \tau_1)
		\\
		&= (-1)^{\lambda(\tau_1,\tau_2)} g_*(\tau_1)
	\end{align*}
	and $g_*(\tau_1)$ is a vertex. If $\tau_1*\tau_2$ is degenerate, then $h(\tau_2 \ot \tau_1) = 0$, and therefore $f(\tau_1 \ot \tau_2)=0$.
\end{proof}

Observe that $\Cyc_r$ acts freely on the vertices of $\partial \simplex^{r-1}$, but it only acts freely on the vertices of $\partial \sd\simplex^{r-1}$ if $r$ is prime. Therefore $r$-cyclic assemblage maps only exist if $r$ is prime. On the other hand, the following lemma, pointed to us by Martin Palmer, assures the existece of such maps when $r$ is prime.

\begin{definition}
	An \emph{$r$-straightening} is a choice, for each non-empty proper subset $\tau$ of $\{0,1,\dots,r-1\}$, of an element $x_\tau\in \tau$. An \emph{$r$-cyclic straightening} is an straightening that is equivariant with respect to the action of the cyclic group $\Cyc_r$. An \emph{$r$-cyclic straightening with duality} is an $r$-cyclic straightening such that the cyclic predecessor of $x_\tau$ is not an element of $\tau$.
\end{definition}

\begin{lemma}\label{lemma:straightening}
	Assemblage maps are in bijection with $r$-straightenings, $r$-cyclic assemblage maps are in bijection with $r$-cyclic straightenings, and $r$-cyclic assemblage maps with duality are in bijection with $r$-cyclic straightenings with duality.
\end{lemma}

\begin{proof}
	An $r$-cyclic straightening $\tau\mapsto x_{\tau}\in \tau$ determines the image of $g$ on vertices, thus
	\[
	g(\tau_0\subset\dots\subset\tau_{k-1}) = (x_{\tau_0},\dots,x_{\tau_{k-1}})
	\]
	where $\pi$ is the permutation that orders the values on the right. To check that it is an assemblage map,
	\[
	g_*\circ s_*([a_0,\dots,a_{k-1}]) = g\left(\sum_{\perm} (-1)^{\sign{\perm}} (a_{\perm(0)}|\dots|a_{\perm(k-1)})\right)
	\]
	vanishes in all summands except for the summand $(\omega_0 \subset \omega_1 \subset \dots \subset \omega_{k-1})$ defined as
	\begin{align*}
		\omega_{k-1} &= a
		&
		\omega_{i} &= \omega_{i+1}\smallsetminus \{g(\omega_{i+1})\}.
	\end{align*}
	The second assertion is evident. For the last one, if $\tau$ is a face of $\sd \partial \simplex^{r-1}$, then $\rho^{-1}(g(\Lambda(\tau))) = \rho^{-1}(x_{\tau^c})$ is the cyclic predecessor of $x_{\tau^c}$, which does not belong to $\tau^c$, and therefore belongs to $\tau$.
\end{proof}

\begin{example}
	If $r=3$, the only assemblage map with duality comes from the choice
	\begin{align*}
		[0]&\mapsto 0 & [0,1]&\mapsto [0]
	\end{align*}
\end{example}

\begin{example}\label{example:asymmetries5}
	If $r=5$, there are four assemblage maps with duality corresponding to the following choices:
	\begin{align*}
		[0]&\mapsto 0 & [0,1]&\mapsto 0 & [0,2]&\mapsto 0 & [0,1,2]&\mapsto 0 & [0,1,3] & \mapsto 0 & [0,1,2,3] & \mapsto 0 \\
		[0]&\mapsto 0 & [0,1]&\mapsto 0 & [0,2]&\mapsto 2 & [0,1,2]&\mapsto 0 & [0,1,3] & \mapsto 0 & [0,1,2,3] & \mapsto 0 \\
		[0]&\mapsto 0 & [0,1]&\mapsto 0 & [0,2]&\mapsto 0 & [0,1,2]&\mapsto 0 & [0,1,3] & \mapsto 3 & [0,1,2,3] & \mapsto 0 \\
		[0]&\mapsto 0 & [0,1]&\mapsto 0 & [0,2]&\mapsto 2 & [0,1,2]&\mapsto 0 & [0,1,3] & \mapsto 3 & [0,1,2,3] & \mapsto 0
	\end{align*}
\end{example}

If $r=7$ there are several thousands of assemblage maps with duality.

\begin{figure}
	\begin{tikzpicture}
		\draw (0,1.732) -- (1,0);
		\draw (-1,0) -- (1,0);
		\draw (-1,0) -- (0,1.732);
		\filldraw [gray] (-.5,.866) circle (2pt);
		\filldraw [gray] (.5,.866) circle (2pt);
		\filldraw [gray] (0,0) circle (2pt);
		\filldraw [black] (0,1.732) circle (2pt) node[anchor=south]{2};
		\filldraw [black] (1,0) circle (2pt) node[anchor=west]{1};
		\filldraw [black] (-1,0) circle (2pt) node[anchor=east]{0};
		\draw[->] (0,0) -- (-.5,0);
		\draw[->] (-.5,.866) -- (-.25,1.299);
		\draw[->] (.5,.866) -- (.75,.433);
	\end{tikzpicture}
	\caption{The $r$-cyclic assemblage map with duality for $\partial \simplex^2$.}
\end{figure}
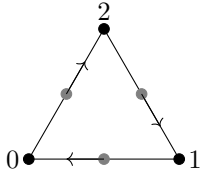

\begin{example}\label{example:f3_1}
	For $r=3$, and $a \ot b = [0] \ot [2,1]$ we have:
	\begin{align*}
		f([0] \ot [2,1])
		&= (-1)^{3\cdot 3}g_*\circ s_*^{\Psd}\circ h_*\circ \twist([0] \ot [2,1])
		\\
		&= -g_*\circ s_*^{\Psd}\circ h_*([2,1] \ot [0])
		\\
		&= g_*\circ s_*^{\Psd}\circ h_*([1,2] \ot [0])
		\\
		&= g_*\circ s_*^{\Psd}((0) \ot [0])
		\\
		&= -g_*((0))
		\\
		&= -[0]
	\end{align*}
	Observe that, since the summand $[1,2]$ appears with positive sign in $\partial \sigma = \partial([0,1,2])$, we have that $f([0] \ot \partial\sigma) = [0]$, as dictated by Condition \eqref{cond:1}.
\end{example}

\begin{example}\label{example:f3_2}
	For $r=3$ and $a \ot b = [0,1] \ot [2,0]$, we have:
	\begin{align*}
		f([0,1] \ot [2,0]) &= (-1)^{3\cdot 4}g_*\circ s_*^{\Psd}\circ h_*\circ \twist([0,1] \ot [2,0])
		\\
		&= g_*\circ s_*^{\Psd}\circ h_*([2,0] \ot [0,1])
		\\
		&= -g_*\circ s_*^{\Psd}\circ h_*([0,2] \ot [0,1])
		\\
		&= g_*\circ s_*^{\Psd}((1) \ot [0,1])
		\\
		&= -g_*\circ s_*^{\Psd}((1) \ot [1,0])
		\\
		&= -g_*((1 \subset 01))
		\\
		&= -[1,0]
		\\
		&= [0,1]
	\end{align*}
	Observe that this exhibits Condition \eqref{cond:1} and Condition \eqref{cond:2}:
	\begin{align*}
		f([0,1] \ot \partial\sigma) &= [0,1]
		&
		f(\partial\sigma \ot [2,0]) &= \rho[2,0].
	\end{align*}
\end{example}

\begin{example}\label{example:f5_1}
	For $r=5$ and the first $r$-cyclic straightening of Example \ref{example:asymmetries5}, if $a \ot b = [2,3,0] \ot [4,1,3]$, we have:
	\begin{align*}
		f([2,3,0] \ot [4,1,3]) &= (-1)^{5\cdot 6}g_*\circ s_*^{\Psd}\circ h_*\circ \twist([2,3,0] \ot [4,1,3])
		\\
		&= -g_*\circ s_*^{\Psd}\circ h_*([4,1,3] \ot [2,3,0])
		\\
		&= g_*\circ s_*^{\Psd}((0,2) \ot [2,3,0])
		\\
		&= g_*\circ s_*^{\Psd}((0,2) \ot [0,2,3])
		\\
		&= g_*((02 \subset 023))
		\\
		&= [0,2]
	\end{align*}
\end{example}

\begin{example}\label{example:f5_2}
	For $r=5$ and the first $r$-cyclic straightening of Example \ref{example:asymmetries5}, if $a \ot b = [1,2,3,4] \ot [0,1,2]$ we have:
	\begin{align*}
		f([1,2,3,4] \ot [0,1,2])
		&= (-1)^{5\cdot 7}g_*\circ s_*^{\Psd}\circ h_*\circ \twist([1,2,3,4] \ot [0,1,2])
		\\
		&= -g_*\circ s_*^{\Psd}\circ h_*([0,1,2] \ot [1,2,3,4])
		\\
		&= -g_*\circ s_*^{\Psd}((3,4) \ot [1,2,3,4])
		\\
		&= -g_*\circ s_*^{\Psd}((3,4) \ot [3,4,1,2])
		\\
		&= g_*((34 \subset 134 \subset 1234)) - g_*((34 \subset 234 \subset 1234))
		\\
		&= [3,3,1] - [3,2,1]
		\\
		&= [1,2,3]
	\end{align*}
	Observe that this exhibits Condition \eqref{cond:2}: $f(\partial\sigma \ot [0,1,2] = [1,2,3])$.
\end{example}

	% !TEX root = ../oddp.tex

\subsection{The bar resolution}\label{ss:milnor}
There is an inclusion $\chains(\EC_r) \to \Omega_*(r)^{\nf}$ that sends a tuple $(a_0,a_1,\dots,a_k)$ to the pieced word $(a_0|a_1|\dots|a_k)$. The composition
\[
\chains(\EC_r) \lra \Omega_*(r)^{\nf} \lra \rW(r)
\]
is an equivariant chain map. This composition has a description analogous to the one used to define $\Psiom$, using the following suspension map:
\[
S(a_0,\dots,a_k) = \begin{cases}
	(-1)^{\sign{\perm}} (a_{0},\dots,a_{k-r+1}) & \text{ if $k \geq r-1$} \\
	\emptyset & \text{if $k = r-2$ and $a_i\neq a_j$ for all $i,j$} \\
	0 & \text{otherwise.}
\end{cases}
\]
where $\pi$ is the permutation that orders $(a_{k-r+1},\dots,a_{k})$. The image of $e_{-q}^{\vee}$ is the sum of all words $A = (a_0,\dots,a_{q-1})\in \cochains[-q] \EC_r$ such that
\begin{itemize}
\item none of the subwords of the form $(a_{q-1-\ell(r-1)},a_{q-1-(\ell-1)(r-1))}$ of length $r$ contains repeated elements,
\item the only subword of the form $(a_0,\dots,a_{q-1-\ell(r-1)})$ of length smaller than $r$ becomes, after having been ordered, an increasing sequence that alternates even and odd entries different from $r-1$, starting with an even entry.
\end{itemize}
each word has the sign needed to order each of the subwords mentioned above, and the coefficient $\frac{\varphi(m)!}{(\tilde{r}!)^j}$ where $m$ is the length of the last of the subwords mentioned, and $j$ is the total number of subwords and $\varphi(m) = \lfloor\frac{r-m-1}{2}\rfloor$.
\begin{example} If $r=3$,
\begin{align*}
	\Psi^\vee(e_{-1}^\vee) &= (0)
	\\
	\Psi^\vee(e_{-2}^{\vee}) &= (0,1)-(1,0)
	\\
	\Psi^\vee(e_{-3}^{\vee}) &= (0,1,2) - (0,2,1)
\\
	\Psi^\vee(e_{-4}^{\vee}) &= (0,1,2,0) - (0,1,0,2) + (1,0,2,1) - (1,0,1,2)
\end{align*}
if $r= 5$,
\begin{align*}
	\Psi^\vee(e^\vee_{-1}) &= \frac{1}{2}\left((0) + (2)\right)
	\\
	\Psi^\vee(e^\vee_{-2}) &= \frac{1}{2}\left((0,1) - (1,0) + (0,3) - (3,0) + (2,3) + (3,2)\right)
	\\
	\Psi^\vee(e^\vee_{-3}) &= \frac{1}{2}\left((0,1,2) -(0,2,1) + (1,2,0) - (1,0,2) + (2,0,1) - (2,1,0)\right)
\end{align*}
$\Psi^\vee(e^\vee_{-4})$ is the signed sum of all permutations of $(0,1,2,3)$, so it has $24$ summands, again with coefficient $\frac{1}{2}$.
\end{example}

This map is different from the one obtained in \cite[Prop.~6.16]{brumfiel2023explicit}, and in fact that map cannot be used to construct $\Psiom$ following the recipy of Definition \ref{def:psiom}.

\subsection{The even prime}\label{ss:even} If $r= 2$ and $R=\bF_2$, then $\Omega_*(r)^{\nf}$ is isomorphic to the unnormalized chain complex of $\EC_2$, while $\rW(r)$ is isomorphic to the normalized chain complex of $\EC_2$. The quotient map $\uchains(\EC_2) \to \chains(\EC_2)$ induces therefore a map
\[
\Omega_*(r)^{\nf} \to \rW(r)
\]
that satisfies the conditions of Lemma \ref{lemma:omegar}.
\begin{remark}
	The combinatorics in pages 13 and 14 of \cite{medina2021fast_sq} translate here to the fact that the map from the unnormalized chains to the normalized chains is well-defined, i.e., to the fact that the degenerate simplices form a subcomplex of the unnormalized chain complex.
\end{remark}
	% !TEX root = ../oddp.tex

\section{Suspension}\label{s:suspension}

In this section, we show that our connected diagonals yield nice connected diagonals on the chain complex of a simplicial $\susp{}$-spectrum or a simplicial $\sus{}$-spectrum. The unstable diagonals of \cite{medina2021may_st} yield nice stable diagonals on the chain complex of a simplicial $\sus{}$-spectrum. At the end we give some comments on other stable categories. We have chosen to work with semi-simplicial pointed sets, but one may very well work with simplicial pointed sets using the Kan suspension instead.

\subsection{Suspension and spectra} The \emph{left suspension functor} sends an augmented simplicial pointed set $X_\bullet$ to a simplicial pointed set $(\Sigma X)_\bullet$ that satisfies the following (as before, we denote by $L'(X)$ the semi-simplicial object that is obtained from $L(X)$ by quotienting by the degeneracies):
\begin{align*}
 \abel'(\susp{} X)_n &\cong \abel'(X)_{n-1},& \d_i(\susp{}\sigma) &= \susp{} \d_{i+1}(\sigma),& d_{0}(\susp{}\sigma) &= *.
\end{align*}
There is also a \emph{right suspension functor} satisfying the following:
\begin{align*}
 \abel'(\sus{} X)_n &= \abel'(X)_{n-1}, &\d_i(\sus{}\sigma) &= \sus{} \d_i(\sigma),& d_{n+1}(\sus{} \sigma) &= *.
\end{align*}
We also have that, for an augmented semi-simplicial pointed set $X_\bullet$ in $\cC$,
\begin{align*}
 \cadenas(\Sigma X)&\cong \susp{} \cadenas(X) & \cadenas(\sus{}X) &\cong \sus{}\cadenas(X)
\end{align*}

An $\susp{}$-spectrum is a sequence of (augmented) semi-simplicial pointed sets $E = \{E_m\}_{m \geq 0}$ with structural maps $\susp{} E_m \to E_{m+1}$. An $\sus{}$-spectrum is a sequence of (augmented) semi-simplicial pointed sets $\{E_m\}_{m \geq 0}$ with structural maps $\sus{} E_m \to E_{m+1}$.

The chain complex of a $\susp{}$-spectrum is the colimit
\[
 \chains(E) = \colim_m \susp{-m} \chains(E_m) \to \susp{-m-1} \chains(E_{m+1})
\]
The chain complex of a $\sus{}$-spectrum is the colimit
\[
 \chains(E) = \colim_m \sus{-m} \chains(E_m) \to \sus{-m-1} \chains(E_{m+1})
\]
These chain complexes compute the ordinary homology of the spectrum $E$ and come equipped with a stable $\mathbb{E}_\infty$ structure \cite{Gill2020}. The stable $\mathbb{E}_\infty$ operad has two well-known models \cite[Appendix]{berger2004combinatorial}: the stable Barratt-Eccles operad or the stable surjection operad, and are defined as the limit of the suspension endomorphism of the Barratt Eccles operad $\cE$ or the surjection operad $\chi$. In arity $r$:
\[\cE_{st} = \lim( \Sigma^2\cE(r)[2-2r] \lra \Sigma \cE(r)[1-r] \lra \cE(r).\]
\[\chi_{st} = \lim( \Sigma^2\chi(r)[2-2r] \lra \Sigma \chi(r)[1-r] \lra \chi(r).\]
Now, these models are clearly not finitely generated (nor free, nor of finite type), so even if the spectrum of which we are taking spectral chains is finite, the model of the action is not finite.

\subsection{A connected $r$-cyclic diagonal for spectra}

Consider the following diagram, where the vertical maps are chain maps of degree $r$ with respect to the right suspension (i.e., they commute with the differential)
\[
\xymatrix{
	\Wd\hotimes \rchains(X)\ar[r]^{\mu}\ar[d]^{\id\otimes \sus{r}} & \chains(X)^{\otimes r} \ar[d]^{\sus{\otimes r}}
	\\
	\Wd\hotimes \rchains(\sus{} X)\ar[r]^{\mu} & \chains(\sus{}X)^{\otimes r}	
}
\]
The left vertical map sends $e_{-q}^\dd\hotimes \tau$ to $e_{-q}^\dd\hotimes \sus{}\tau$, and notice that $\binom{r}{2}(n+1) \equiv \sum_s s(n+1)$. The right vertical map sends $\tau_0\otimes \ldots\otimes \tau_{r-1}$ to $\sus{}\tau_0\otimes \ldots\otimes \sus{}\tau_{r-1}$ with the sign $(-1)^{\sum_s s|\tau_{s}|}$. 
\begin{lemma} The lower composition equals $(-1)^{q+\binom{r}{2}(n+1)}\tilde{r}!$ times the upper composition.
\end{lemma}
\begin{proof}
The summands in the image of both horizontal maps are obtained from the same pairs $(U,A)$, and the signs associated to this pair are all independent of the dimension of $\tau$ except for:
\begin{enumerate}
	\item the sign of Lemma \ref{lemma:2},
	\item the sign of Lemma \ref{lemma:3},
	\item the sign of $\eta$ given in equation \ref{eq:eta},
	\item the factor $(\tilde{r}!)^{n+1}$ of Lemma \ref{lemma:omegar}.
\end{enumerate}
The sign of Lemma \ref{lemma:3} equals the product of the sign of the right vertical maps and $(-1)^{\binom{r}{2}(n+1)}$. The sign of Lemma \ref{lemma:2} together with the sign of $\eta$ simplifies to the sign $\lambda(U^c,U)$. The difference between the two horizontal arrows is then that the complement $U^c$ in the upper arrow is computed in $\{0,1,\ldots,n\}$ while in the lower arrow is computed in $\{0,1,\ldots,n+1\}$. The difference yields that the diagram commutes up to $(-1)^q(\tilde{r}!)$. 
\end{proof}

Consider now the following diagram, where the vertical maps are chain maps of degree $r$ with respect to the left suspension (i.e., they anticommute with the differential)
\[
\xymatrix{
	\Wd\hotimes \rchains(X)\ar[r]^{\mu}\ar[d]^{\id\otimes \susp{r}} & \chains(X)^{\otimes r} \ar[d]^{\rho\susp{\otimes r}}
	\\
	\Wd\hotimes \rchains(\susp{} X)\ar[r]^{\mu} & \chains(\susp{}X)^{\otimes r}	
}
\]
The left vertical map sends $e_{-q}^\dd\hotimes \tau$ to $(-1)^{q}e_{-q}^\dd\hotimes \susp{}\tau$. The right vertical map sends $\tau_0\otimes \ldots\otimes \tau_{r-1}$ to $\susp{}\tau_{r-1}\otimes \susp{}\tau_0\ldots\otimes \susp{}\tau_{r-2}$ with the sign needed to reorder the $\tau$'s and to move an operator of degree $1$ from the left to each factor. 
\begin{lemma}
	The lower composition equals $(-1)^{\binom{r}{2}(n+1)}\tilde{r}!$ times the upper composition.
\end{lemma}
\begin{proof}
	The argument is as before, but observing that there is a bijection between the terms labeled by $(U,A) = ((u_0,\ldots,u_{q-1}),(a_0,\ldots,a_{q-1}))$ above and those labeled by $(U',A) = ((u_0+1,\ldots,u_{q-1}+1),(a_0,\ldots,a_{q-1}))$ below. Again, the two vertical signs account for the difference of the signs obtained in the upper and lower horizontal arrows for Lemma \ref{lemma:3}. In this case, $\lambda(U^c,U) \equiv \lambda((U')^c,U')$, hence the diagram commutes. 
\end{proof}
From these two lemmas we deduce:
\begin{proposition}\label{prop:suspensionconnected}
 The effect of left and right suspensions on the $r$-cyclic connected diagonals defined in Corollary \ref{thm2:mainthm} is the following:
 \begin{align*}
  \mu(e_{-q}^\dd \ot \sus{}\tau) &= (-1)^{q + \binom{r}{2}(n+1)}\tilde{r}!\sus{\ot r}\mu(e_{-q}^\dd \ot \tau)
  \\
 \mu(e_{-q}^\dd \ot \susp{}\tau) &= (-1)^{q + \binom{r}{2}(n+1)}\tilde{r}!\susp{\ot r}\rho (\mu(e_{-q}^\dd \ot \tau) 
 \end{align*}
\end{proposition}

\begin{corollary}
 On the stable chains of a $\sus{}$-spectrum $\{E_m\}_{m \geq 0}$ we can define the following connected diagonal $\mu_{\mathrm{sp}}$ (which restricts to a stable diagonal):
 \[
 \mu_{\mathrm{sp}}(e_{-q}^\dd \hotimes \tau) = (-1)^{mq+\binom{r}{2}\left(m(m+n-1) + \binom{m+1}{2}\right)} \frac{1}{(\tilde{r}!)^m}\mu(e_{-q}^\dd \hotimes \tau)
 \]
 where $\tau\in\chains(E_m)$.
\end{corollary}

\begin{corollary}
 On the stable chains of a $\susp{}$-spectrum $\{E_m\}_{m \geq 0}$ we can define the following connected diagonal $\mu_{\mathrm{sp}}$ (which restricts to a stable diagonal)
 \[
 \mu_{\mathrm{sp}}(e_{-q}^\dd \hotimes \tau) = (-1)^{mq+\binom{r}{2}\left(m(m+n-1) + \binom{m+1}{2}\right)}\frac{1}{(\tilde{r}!)^{m}}\rho^{-m} \mu(e_{-q}^\dd \hotimes \tau)
 \]
 where $\tau\in \chains(E_m)$.
\end{corollary}

\subsection{A stable $r$-cyclic diagonal for right spectra} In the next proposition we observe that the homomorphism $W_*(r) \to \cE(r)$ constructed in \cite{medina2021may_st} is compatible, up to a sign to be determined, with the right suspension, and therefore gives a nice map
\[
 W_*(r)_{st} \lra \cE_{st}
\]
that yields a stable diagonal on the chain complex of a $\sus{}$-spectrum.

\begin{proposition}\label{prop:suspensionunstable}
 The effect of left and right suspensions on the $r$-cyclic unstable diagonals of \cite{medina2021may_st} is the following:
 \begin{align*}
 \mu(e_q \ot \susp{} \tau) &= (-1)^{q+\binom{r}{2}n}(\tilde{r}-1)!\susp{\ot r}\sum_{i=1}^{\tilde{r}} \rho^{2i}\mu(e_{q-r+1} \ot \tau)
 \\
 \mu(e_q \ot \sus{}\tau) &= (-1)^{q+\binom{r}{2}n}\tilde{r}!\sus{\ot r}\mu(e_{q-r+1} \ot \tau)
 \end{align*}
\end{proposition}

\begin{proof}
 The effect of suspension on the Barrat-Eccles operad and the surjection operad is described in the appendix of \cite{berger2004combinatorial}. From that description one deduces that the summands on the left corresponding to an element $(\rho^0,\rho^{i_1},\rho^{i_1+1},\dots,\rho^{i_k},\rho^{i_k+1})$ will vanish unless the first $r$ elements are all distinct. A count yields the result.

 For the second statement, we follow the same argument, except that the summands that do not vanish are those whose last $r$ elements are all distinct.
\end{proof}

 If $\mu$ is the unstable diagonal of \cite{medina2021may_st}, there is a stable diagonal $\mu_{st}$ on the stable chains of a $\sus{}$-spectrum with coefficients in $\bZ[\frac{1}{\tilde{r}!}]$ defined as
 \[
 \mu_{st}(e_q \ot \tau) = (-1)^{mq+\binom{r}{2}\left(m(m+n-1) + \binom{m+1}{2}\right)}\frac{1}{(\tilde{r}!)^{m}}\mu(e_{q+m(r-1)}) \ot \tau)
 \]
 where $\tau\in \chains[n](E_m)$. Since $\sum_{i=1}^{\tilde{r}} \rho^{2i}$ is invertible, we have can derive a formula for the left suspension too.

 The disparity in the behaviour with respect to left suspension of the operations in this work and the ones in \cite{medina2021may_st} show that

 \begin{corollary}
	The chain level operations of \cite{medina2021may_st} are not the same as the ones of this work if $r\geq 5$.
\end{corollary}

	%\input{sec/8products}
	% !TEX root = ../oddp.tex

\section{Kan spectra}\label{s:9Kanspectra}

The category of Kan spectra is one of the many categories that model spectra. Introduced by Kan in \cite{Kan1963} and later developed by several articles \cite{burghelea_kanspectraI1967, burghelea_kanspectraII_1968, burghelea_kanspectraIII1969,Brown1973}. It has recently received some attention \cite{Stephan2015, CKP2023}.

\begin{definition}
	The \emph{stable simplex category} $\kansimplex$ is the colimit of the functors
	\[
	\simplex \lra \simplex \lra \simplex \lra \dots
	\]
	that send an ordinal $[n]$ to the ordinal $[n+1]$ and extend a map $f \colon [n] \to [m]$ to a map $f' \colon [n+1] \to [m+1]$ by setting $f'(n+1) = m+1$. It has one object $[n]$ for each integer $n\in \bZ$, and there are morphisms $\d^i \colon [n-1] \to [n]$ and $\s^i \colon [n] \to [n-1]$, for all $i \geq 0$, satisfying the simplicial identities.
\end{definition}

The morphisms from $[n]$ to $[m]$ are in bijection with the endomorphisms $f$ of the first infinite ordinal $\omega$ such that
\begin{itemize}
	\item $|f^{-1}(k)|<\infty$ for all $k\in \omega$,
	\item $|f^{-1}(k)|\neq 1$ for finitely many $k\in \omega$,
	\item $\sum_k (1-f^{-1}(k)) = m-n$.
\end{itemize}

\begin{definition}
	A \emph{Kan spectrum} is a presheaf $X \colon \kansimplex^{\op} \to \Setp$ on pointed sets such that for each $x\in X_n$ there is an $m$ such that $d_i(x) = *$ for all $i>m$. A map of Kan spectra is a natural transformation.
\end{definition}

This condition allows to define chain complexes $\ucadenas(X)$ and $\cadenas(X)$ as for simplicial sets: In degree $n$ they are generated by the pointed set $X_n$ modulo the base point, and the differential is the alternate sum of the face maps.

The standard augmented infinite simplex $\asimplex$ is the Kan spectrum with one $n$-simplex for each order-preserving map $[n] \to [\omega]$.

Write as before $\uchains(X)$ and $\chains(X)$ for their shifted versions. There is a covariant functor $\kansimplex \to \Ch{R}$ that sends any object $[n]$ to the chain complex of the augmented infinite simplex $\chains(\asimplex^\infty)$, and the action of faces and degeneracies are the stabilization of the faces and degeneracies of the finite simplices. There is another covariant functor $\kansimplex \to \Ch{R}$ that sends any object $[n]$ to the cochains $\cochains(\asimplex^\infty)$, and the action of the faces and degeneracies is the limit of their finite counterparts (note that this second complex is not free). Then, we have, as in Lemma \ref{lemma:2}:

\begin{lemma}
	There is a natural isomorphism
	\[
	\chains(\asimplex^{\infty}) \ot_{\kansimplex} \rA(X)_\bullet \cong \cadenas(X)
	\]
\end{lemma}

Then, the rest of \cref{s:3complexes} remains true: One obtains a stable cosimplicial cochain complex $\chains(\asimplex^\infty)$ and a stable cosimplicial cochain complexes $\Omd(r,\infty)$ and $\Omhatd(r,\infty)$, and the former is the dual to a chain complex $\Om(r,\infty)$ that admits maps $\eta^{n}$ to $\Om(r)$. With these changes, Proposition \ref{prop:omegarm} yields a connected diagonal on the chain complex of a Kan spectrum.

Note that the previous lemma does not make sense without dualizing: the top simplex in $N_*(\asimplex^\infty)$ is far away, at infinite dimension.
	% !TEX root = ../oddp.tex

\section{Explicit formulae}\label{s:formulas}

Let $r$ be an odd prime and let $\tau\mapsto g(\tau)$ be a $r$-cyclic straightening with duality. Let $\chains(X)$ denote the right suspension of the normalized chain complex $\cadenas(X)$ of $X$. If $x$ is a cochain of dimension $n$ in the cochain complex of an augmented simplicial set $X$, then $\power^i([x]) = [y]$ is a cochain of dimension $n = m+i$ with
\[
	y(\tau) = \frac{1}{(\tilde{r}!)^{m}}x^{\ot r}(\mu(e_{ri} \hotimes \tau)).
\]
The element $\mu(e_{q} \hotimes \sigma)\in \chains(X)^{\ot r}$ is computed as follows: Let $\cP(r,q)$ be the collection of pairs $(U,A)$ such that
\renewcommand{\theenumi}{\roman{enumi}}%
\begin{enumerate}
	\item $U = (u_0,\dots,u_{q-1})$ is a non-decreasing sequence with $u_j\in \{0,1,\dots,n\}$,
    \item $A = (a_0,\dots,a_{q-1})$ is a sequence with $a_i\in \{0,1,\dots,r-1\}$, and
	\item\label{condfor:3} $a_i<a_j$ if $u_i=u_j$.
\end{enumerate}
Each pair $(U,A)\in \cP(r,q)$ and each chain $\tau\in \chains(X)$ of dimension $n$ (and degree $n+1$) defines the following generator of $\chains(X)^{\ot r}$:
\[
	d_{U_A^{r-1}}(\tau) \ot d_{U_A^{r-2}}(\tau) \ot \dots \ot d_{U_A^{0}}(\tau)\in \chains(X)^{\ot r}.
\]
where $U_A^k = \{u_i\mid u_i + a_i \equiv k\mod r\}$ for each $j\in 0,1,\dots,r-1$. We will now define recursively the coefficient $\nu_{U,A}$ of this generator. Let $A'$ be the sequence $\{a_i+u_i\mod r\mid i=0,\dots,q-1\}$ and let $q_k$ be the amount of elements of $A'$ that are equal to $k$, so $q=\sum q_k$, and define
\renewcommand{\theenumi}{\arabic{enumi}}
\begin{enumerate}
 \item $\signo_0 = nq$,
	\item $ \signo_1 = \sum_k\binom{q_k}{2}$,
	\item $ \signo_2$ is the sum of all entries of $U$
	\item $ \signo_3$ is $(n+1)$ times the sum of all the entries of $A'$,
	\item $ \signo_4$ is the parity of the permutation that orders $A'$.
\end{enumerate}
Define $\signo =\signo_0+\signo_1+\signo_2+\signo_3+\signo_4$. If there is a subsequence $u_j = u_{j+1} = \dots = u_{j+r-1}$ of length at least $r$, then $\nu_{U,A}=0$, otherwise:
\begin{itemize}
	\item If $q<r$, let $A' = (a'_0,\dots,a'_{q-1})$ be the result of ordering $A$ with $\signot$ the parity of that permutation.
	\begin{itemize}
		\item[$\bullet$] If $a_{2j}'$ is even and different from $r-1$ and $a_{2j+1}'$ is odd, then define $\nu_{U,A} = (-1)^{\signot}\varphi(q)!(\varphi(0)!)^n$, where $\varphi(q) = \lfloor \frac{r-1-q}{2}\rfloor$.
		\item[$\bullet$] Otherwise, $\nu_{U,A}=0$.
	\end{itemize}
\item If $q \geq r$, let $\ell$ be the biggest number smaller than $r$ such that $u_{q-\ell-1}<u_{q-\ell}$. Let $k$ be such that $u_{q-k-1}<u_{q-k} = u_{q-\ell-1}$. Define
\begin{align*}
w &= (a_{q-\ell},a_{q-\ell+1},\dots,a_{q-1})
\\
z &= (a_{q-k},\dots,a_{q-\ell-1})
\end{align*}
\begin{itemize}
	\item[$\bullet$] If $\{0,1,\dots,r-1\}\not \subset w\cup z$, then $\nu_{U,A} = 0$.
	\item[$\bullet$] Otherwise, define
\begin{itemize}
	\item $x = (x_0,\dots,x_{r-\ell-1})$, the ordered complement of $w$ in the set $\{0, \dots ,r-1\}$,
	\item	$y = (y_0,\dots,y_{k-r-1})$, the ordered complement of $x$ in $z$,
	\item $\signot_0$ the parity of the permutation that orders $w\cup x$
	\item $\signot_1$ the parity of the permutation that reorders $x\cup y$ to $z$.
	\item $\signot_2=(|x|-1)|y|$.
\end{itemize}
and for each permutation $\pi$ of the entries of $y$, let $\signot_3$ be its parity and define the sequence
\[
\quad\quad \omega_\pi = (g(x),g(x\cup y_{\pi(0)}), g(x \cup y_{\pi(1)}\cup y_{\pi(1)}), \dots, g(x\cup y)).
\]
\begin{itemize}
	\item[$\bullet$] If $\omega_\pi$ has repetitions, then set $\nu_{U,A,\pi} = 0$.
	\item[$\bullet$] Otherwise, let $(\omega'_0,\dots,\omega'_{k-r})$ be the result of ordering the sequence and let $\signot_4$ be the parity of that reordering. Define a new pair $(U_{\pi},A_{\pi})$ as
	\begin{itemize}
		\item[] $\quad U_{\pi} = (u_0,\dots,u_{q-r})$
		\item[] $\quad A_{\pi} = (a_0,\dots,a_{q-k-1},\omega'_0,\dots,\omega'_{k-r})$.
	\end{itemize}
Define $\signot_\pi = \signot_0+\signot_1+\signot_2+\signot_3+\signot_4$.
\end{itemize}
Finally, define
\[
	\nu_{U,A} = \sum_{\pi}(-1)^{\signot_\pi}\nu_{U_\pi,A_\pi}.
\]
\end{itemize}
\end{itemize}
The coefficient of $(U,A)$ as $(-1)^{\signo}\nu_{U,A}$.

\begin{example}\label{example:formulas1_3}
	If $r=3$, $X= \asimplex^7$, $U = (0,0,1,3,4,6), A = (0,1,2,0,2,1)$ and $\tau = [0,1,2,3,4,5,6,7]$ , we have:
	\begin{align*}
		\signo_0 &\equiv 0
		&
		\signo_1 &\equiv 1
		&
		\signo_2 &\equiv 0
		&
		\signo_3 &\equiv 0
		&
		\signo_4 &\equiv 1
		&
		\signo &\equiv 0
	\end{align*}
	To compute $\nu_{U,A}$, in the first iteration we obtain that $x=(0)$, $y=\emptyset$, thus there is only one permutation $\pi$ to be considered, for which
	\begin{align*}
		\signot_0 &\equiv 1
		&
		\signot_1 &\equiv 0
		&
		\signot_2 &\equiv 0
		&
		\signot_3 &\equiv 0
		&
		\signot_4 &\equiv 0
		&
		\signot_{\pi} &\equiv 1
	\end{align*}
	The new pair is $U_{\pi} = (0,0,1,3), A_\pi = (0,1,2,0)$ and $\nu_{U,A} = -\nu_{U_\pi,A_\pi}$. Renaming $U'= U_\pi$ and $A'=A_\pi$, in the second iteration we obtain that $x=(1)$, $y=(0)$, thus there is only one permutation $\pi$ to be considered, for which
	\begin{align*}
		\signot_0 &\equiv 0
		&
		\signot_1 &\equiv 1
		&
		\signot_2 &\equiv 0
		&
		\signot_3 &\equiv 1
		&
		\signot_4 &\equiv 0
		&
		\signot_{\pi} &\equiv 0
	\end{align*}
	The new pair is $U'_{\pi} = (0,0), A'_\pi = (0,1)$, and $\nu_{U',A'} = \nu_{U'_\pi,A'_\pi}$. Renaming $U'' = U'_\pi$ and $A'' = A'_\pi$, we compute that $\nu_{U'',A''} = 1$. Therefore the coefficient of the summand given by $(U,A)$ is
	\[
		(-1)^{\signo} \nu_{U,A} = -\nu_{U',A'} = -\nu_{U'',A''} = -1
	\]
\end{example}

\begin{example}
	If $r=5$, $X= \asimplex^2$, $U = (0,1,1,1,2,2,2), A = (1,2,3,0,1,3,4)$ and $\tau = [0,1,2]$ , we have:
	\begin{align*}
 \signo_0 &\equiv 0
		&
		\signo_1 &\equiv 0
		&
		\signo_2 &\equiv 1
		&
		\signo_3 &\equiv 1
		&
		\signo_4 &\equiv 1
		&
		\signo &\equiv 1
	\end{align*}
	To compute $\nu_{U,A}$, in the first iteration we obtain that $x=(0,2)$, $y=(3)$, thus there is only one permutation $\pi$ to be considered, for which
	\begin{align*}
		\signot_0 &\equiv 1
		&
		\signot_1 &\equiv 0
		&
		\signot_2 &\equiv 1
		&
		\signot_3 &\equiv 0
		&
		\signot_4 &\equiv 0
		&
		\signot_{\pi} &\equiv 0
	\end{align*}
	The new pair is $U_{\pi} = (0,1,1), A_\pi = (1,0,2)$ and $\nu_{U,A} = \nu_{U_\pi,A_\pi}$. Renaming $U' = U_\pi$ and $A' = A_\pi$, we have $\nu_{U',A'} = -2^2$. Therefore the coefficient of the summand given by $(U,A)$ is
	\[
		(-1)^{\signo} \nu_{U,A} = -\nu_{U',A'} = 2^2.
	\]
\end{example}

\subsection{The case \texorpdfstring{$r=3$}{r = 3}} If $r=3$, there is a simpler description of the operations. First, we replace Condition \eqref{condfor:3} for the following:
\begin{itemize}
	\item[(iii')] $a_i\neq a_{i+1}$ for all $i=0,\dots,q-1$.
\end{itemize}
Define the \emph{regular blocks} of $A$ as the subsequences $(a_{q-2k-1},a_{q-2k},a_{q-2k+1})$ of length $3$. Define the \emph{exceptional block} of $A$ as $a_0$ if $q$ is odd and $(a_0,a_1)$ if $q$ is even. A block is \emph{ascending} (resp. descending) if it has no repeated entries and it is cyclically ordered (resp. not cyclically ordered). The coefficient $\nu_{U,A}$ is non-zero if and only if
\begin{enumerate}
	\item no block has repeated entries,
	\item if a block $(a_{q-2k-1},a_{q-2k},a_{q-2k+1})$ is descending, then $u_{q-2k-1}<a_{q-2k}<a_{q-2k+1}$,
	\item if a block $(a_{q-2k-1},a_{q-2k},a_{q-2k+1})$ is ascending, then $u_{q-2k-1}\neq a_{q-2k}$ or $a_{q-2k}\neq a_{q-2k+1}$,
	\item The exceptional block is either of the following: $(0),(0,1),(1,0)$.
\end{enumerate}
If that is the case, define $\signo$ as in the previous section and let $\signot$ be the number of descending blocks. Then the coefficient of $(U,A)$ is $\nu_{U,A} = (-1)^{\signo+\signot}$.

\begin{example}
	If $r=3$, $X= \asimplex^7$, $U = (0,0,1,3,4,6), A = (0,1,2,0,2,1)$ and $\tau = [0,1,2,3,4,5,6,7]$ as in Example \ref{example:formulas1_3}, we have the exceptional block $(0,1)$ and the regular blocks $(1,2,0)$ and $(0,2,1)$. Only the last one is descending, thus $\signot = 1$ and $\nu_{U,A} = (-1)^{\signo+\signot} = -1$.
\end{example}

\subsection{Explanation of the formulas for any prime number \texorpdfstring{$r$}{r}} Here we justify the formulas of this section using the main theorem of the paper. A complete read of the paper is necessary to follow this section. Let $\tau\in \chains(X)$ be a chain of dimension $n$ (hence degree $n+1$). Let $(U,A)$ be a generator of $\chains(\asimplex^n\times \EC_r)$, and let $(U',A') = \beta(U,A)$, let $q_k$ be the number of elements of $A'$ that equal $k$ and let $q = \sum_{k}q_k$ the degree of $(U,A)$.

The sign of the map $\alpha$ is the parity of
\[
	\sum_{k}\lambda((U_A^{k})^c,U_A^k) + (n+1)\sum_k q_k + (n+1)\sum_kk(r-1-q_k),
\]
where the first term is the sign of $\Lambda$, the second is the sign of the functor tensor product and the third is the sign of the reordering of the tensor factors. Notice that we have to reorder $(U^{r-1}_A \ot \dots U^0_A) \ot \tau \ot \dots \ot \tau$, hence the sign is $\sum_k(r-1-q_k)$ instead of $\sum_k q_k$. By Remark \ref{remark:alex}, we can write these signs as:
\begin{align}
\label{signo:lambda}	\lambda((U_A^k)^c,U^k_A) &\equiv \sum_i u_i + \sum_k(n+1-q_k)q_k + \binom{q_k}{2}
\\
\label{signo:prod} (n+1)q
\\
 \label{signo:reord} \sum_kk(r-1-q_k)&\equiv (n+1)\sum_{i} a'_i
\end{align}
The sign of the map $\beta$ is the sign that orders $A' = \{a_i+u_i\mid i=0,\dots,q-1\}$. The map $\gamma$ does not have sign and the map $\lambda$ has the sign $(n+1)q$. After the natural cancellations, we obtain $\signo$.

For the map $f$, we have first the sign of turning the right suspension into left suspension $(-1)^{(|w|+|z|)r}$, which can be written as $(-1)^{|y|+1}$. Then, the sign of $\twist$ is the parity of $(r-|x|)(|x|+|y|) \equiv \signot_2$, the sign of $\Lambda$ is $\signot_0$ and the sign of $s_*^{\Psd}$ is the parity $\signot_0$ of the permutation that arranges $x\cup y$ to $z$ plus $|y|+1$. Altogether yield $\signot_0+\signot_1+\signot_2$. For the map $\Phi$ we follow Remark \ref{remark:phidual}.
\subsection{Explanation of the formulas for \texorpdfstring{$r=3$}{r = 3}} A pieced word $A = (A_0\barra\dots\barra A_j)$ admits many underlying words $(a_0,\dots,a_{q-1})$, because the entries of each piece may be permuted. The formulas for $r=3$ follow from the following observation: each pieced word $A$ has a preferred underlying word $A = (a_0,\dots,a_{q-1})$ obtained by ordering each piece of size two so that the second element is the cyclic succesor of the first one. Moreover, with this choice we have that the preferred underlying word of $S(A_0\barra\dots\barra A_{j})$ is always the result of removing the last two elements from $A$, with positive sign unless the last block is descending (this implies that the last two pieces must be singletons).

To check this last statement it suffices to check that
\begin{align*}
	f([0] \ot [1,2]) &= [0]
	&
	f([0,1] \ot [2,0]) &= [0,1]
	\\
		f([0,1] \ot [2]) &= [0]
	&
	f([0,1] \ot [2,1]) &= 0
\end{align*}
since all the other cases are permutations of these three.

	%	\appendix
	%	\input{sec/cup_product}
	%	\input{sec/sym_multiplications}
	%	\input{sec/considerations}
	%	\input{sec/reverse_order}
	%	\input{sec/surjection}
	%	\input{sec/search}
	\sloppy
	\printbibliography
%	\newpage
%	\todos
\end{document}